\documentclass[11pt]{amsart}
\usepackage{amssymb}
\usepackage{amscd}
\usepackage{amsmath}
\usepackage[all]{xy}
\usepackage{mathrsfs}
 \usepackage{color}

\usepackage[dvipdfmx]{hyperref}

\usepackage{bm}

\calclayout

\theoremstyle{plain}
\newtheorem{theorem}{Theorem}[section]
\newtheorem{theorem*}{Theorem}
\newtheorem{proposition}[theorem]{Proposition}
\newtheorem{lemma}[theorem]{Lemma}
\newtheorem{corollary}[theorem]{Corollary}

\theoremstyle{remark}
\newtheorem{remark}[theorem]{Remark}

\theoremstyle{definition}
\newtheorem{definition}[theorem]{Definition}
\newtheorem{example}[theorem]{Example}
\newtheorem{hypothesis}[theorem]{Hypothesis}

\newtheorem*{acknowledgments}{Acknowledgements}

  1

\DeclareMathOperator{\Gal}{Gal}

\DeclareMathOperator{\Hom}{Hom}

\DeclareMathOperator{\im}{im}
\DeclareMathOperator{\coker}{coker}

\newcommand{\TT}{\mathbb{T}}

\newcommand{\QQ}{\mathbb{Q}}

\newcommand{\cF}{\mathcal{F}}

\newcommand{\cI}{\mathcal{I}}

\newcommand{\cL}{\mathcal{L}}

\newcommand{\cN}{\mathcal{N}}
\newcommand{\cO}{\mathcal{O}}
\newcommand{\cP}{\mathcal{P}}
\newcommand{\cQ}{\mathcal{Q}}
\newcommand{\cR}{\mathcal{R}}
\newcommand{\cS}{\mathcal{S}}
\newcommand{\cT}{\mathcal{T}}

\newcommand{\fd}{\mathfrak{d}}

\newcommand{\fq}{\mathfrak{q}}

\newcommand{\fn}{\mathfrak{n}}
\newcommand{\fm}{\mathfrak{m}}
\newcommand{\fz}{\mathfrak{z}}

\newcommand{\FF}{\mathbb{F}}

\newcommand{\NN}{\mathbb{N}}

\newcommand{\ZZ}{\mathbb{Z}}

\newcommand{\sF}{\mathscr{F}}

\newcommand{\rgamma}{\mathbf{R}\Gamma}
\newcommand{\rhom}{\mathbf{R}\Hom}
\newcommand{\lotimes}{\otimes^{\mathbf{L}}}

\makeatletter
 
  \@addtoreset{equation}{subsection}
\makeatother

\begin{document}

\title{On Selmer complexes, Stark systems and derived $p$-adic heights}

\author{Daniel Macias Castillo and Takamichi Sano}

\begin{abstract}
We develop the theory of Nekov\'a\v{r}'s Selmer complexes. 
We prove that, under mild hypotheses, Nekov\'a\v{r}'s Selmer complexes are canonically quasi-isomorphic to ``Poitou-Tate complexes", which arise from Poitou-Tate global duality exact sequences. We give two applications. Firstly, we prove that the determinant of a Selmer complex is canonically isomorphic to the module of Stark systems and, by using this result, we construct a canonical ``Heegner point Stark system" which controls Selmer groups. Secondly, we prove that the derived $p$-adic height pairing of Bertolini-Darmon concides with that of Nekov\'a\v{r}. 
\end{abstract}

\address{Instituto de Ciencias Matem\'aticas, Calle Nicol\'as Cabrera 13-15, Campus de Cantoblanco, 28049 Madrid (Spain)}
\email{daniel.macias@icmat.es}

\address{Osaka Metropolitan University,
Department of Mathematics,
3-3-138 Sugimoto\\Sumiyoshi-ku\\Osaka\\558-8585,
Japan}
\email{tsano@omu.ac.jp}

\maketitle

\tableofcontents

\section{Introduction}

\subsection{Background}

Selmer complexes, introduced by Nekov\'a\v{r} \cite{nekovar}, unify various operations in Iwasawa theory in a systematic way. For example, Mazur's control theorem is interpreted by the language of Selmer complexes in a neat way (see \cite[\S 8.10]{nekovar}). Cassels-Tate pairings and $p$-adic height pairings are also treated systematically by using Selmer complexes and homological algebra (see \cite[\S\S 10 and 11]{nekovar}). In particular, a formalism of ``higher height pairings" arises naturally from this approach (see \cite[\S 11.5]{nekovar}).

Furthermore, as implicitly observed in \cite[Footnote (8) on page 28]{nekovar}, Selmer complexes provide a unified framework for formulating Iwasawa main conjectures. To give an example of the importance of this observation we note that, by using natural formulations of anticyclotomic main conjectures in terms of Selmer complexes (as in \cite{ks}), the second author \cite{sanoderived} is able to establish new descent results. Note also that, in the recent work \cite{cs}, such a reformulation of Iwasawa main conjectures in terms of Selmer complexes  plays a key role in proving new cases of the ``refined Kurihara/Kolyvagin conjectures".

Moreover, it has become clear in the last few years that Selmer complexes are essential to the formulation and study of equivariant leading term conjectures. Indeed, from an equivariant perspective, ``finite-support cohomology complexes" (whose cohomology modules canonically identify with classical or Bloch-Kato Selmer groups and are themselves important examples of Selmer complexes) are not always ``perfect" (i.e., isomorphic in the derived category to a bounded complex of finitely generated, projective modules) and thus need to be approximated by more general Selmer complexes. See for instance the discussion in \S 2 of  \cite{BMC}. In loc. cit., Selmer complexes for abelian varieties defined over number fields are used to formulate and study a general equivariant refinement of the Birch and Swinnerton-Dyer conjecture. In \cite{BB}, a theory of Euler systems that are valued in the cohomology groups of Selmer complexes of $p$-adic representations (rather than just in their Galois cohomology groups) is developed, and allows the authors to prove important new cases of the equivariant Tamagawa number conjecture. In \cite{BM}, the idea of approximating finite-support cohomology complexes by more general Selmer complexes plays a key role in the proof of a substantial part of the ``refined conjectures of the Birch and Swinnerton-Dyer type" of Mazur and Tate. 


In this article, we aim to develop the theory of Selmer complexes. More precisely, we study the following topics. 
\begin{itemize}
\item[(a)] Comparison between Selmer complexes and ``Poitou-Tate complexes" (which naturally arise from Poitou-Tate global duality exact sequences).
\item[(b)] Relation between Selmer complexes and ``Stark systems". 
\item[(c)] Comparison between the ``derived $p$-adic height pairing" of Bertolini-Darmon \cite{BDMT}, \cite{BD der} and the ``higher height pairing" of Nekov\'a\v{r} \cite[\S 11.5]{nekovar}.
\end{itemize}
We shall explain details of these topics in \S\S \ref{s2}, \ref{s3} and \ref{s4} respectively. 

\subsection{Selmer and Poitou-Tate complexes}\label{s2}

Let us briefly recall the definition of Selmer complexes. Let $p$ be a prime number. For simplicity, we assume that $p$ is odd throughout this article. Let $K$ be a number field. Let $S_p$ and $S_\infty$ be the sets of $p$-adic and infinite places of $K$ respectively. Fix a finite set $S$ of $K$ which contains $S_p\cup S_\infty$ and set $S_f:=S\setminus S_\infty$. Let $G_{K,S}$ be the Galois group of the maximal extension $K_S/K$ unramified outside $S$. Let $T$ be a free $\ZZ_p$-module of finite rank endowed with a continuous linear action of $G_{K,S}$. A ``Greenberg local condition"
$$\sF=(\sF_v)_{v\in S_p}$$
is a collection of exact sequences of $G_{K_v}$-representations:
$$\sF_v:0\to T_v^+\to T\to T_v^-\to 0.$$
Nekov\'a\v{r}'s Selmer complex attached to $\sF$ is defined by
$$C_{\rm Nek}:=\widetilde \rgamma_\sF(K,T) := {\rm Cone}\left(\rgamma(G_{K,S},T)\to \bigoplus_{v\in S_p}\rgamma(K_v,T_v^-)\oplus \bigoplus_{v\in S_f\setminus S_p}\rgamma_{/{\rm ur}}(K_v,T)\right)[-1].$$
(See Definition \ref{def sel}.) 

On the other hand, we can also consider Selmer groups defined in a classical way. It is convenient to consider Selmer structures in the sense of Mazur-Rubin \cite{MRkoly}. 
We fix a positive integer $n$ and set
$$A=A_n:=T/p^n. $$
(We often abbreviate $M/p^nM$ to $M/p^n$ for any abelian group $M$.) 
For the Greenberg local condition $\sF$, we naturally define a Selmer structure $\cF$ on $A$ by
$$H^1_\cF(K_v,A):=\begin{cases}
\im(H^1(K_v,A_v^+) \to H^1(K_v,A)) &\text{if $v\in S_p$,}\\
H^1_f(K_v,A) &\text{if $v\notin S_p$},
\end{cases}$$
where we set $A_v^+:=T_v^+/p^n $ and $H^1_f(K_v,A)$ denotes the Bloch-Kato local condition as in \cite[\S 1.3]{R}. The Selmer group for $\cF$ is defined by 
$$H^1_\cF(K,A):=\ker\left(H^1(G_{K,S},A)\to \bigoplus_{v\in S_f}H^1_{/\cF}(K_v,A)\right).$$
One can also consider the dual Selmer structure $\cF^\ast$ on the Kummer dual $A^\ast(1)$ of $A$, and the corresponding Selmer group $H^1_{\cF^\ast}(K,A^\ast(1))$. 

One naturally expects that $H^1_\cF(K,A)$ is related with $H^1$ of $C_{\rm Nek}/p^n:=C_{\rm Nek}\lotimes_\ZZ \ZZ/p^n$. In general, these do not coincide. However, under mild hypotheses, we have a canonical map
$$\kappa_1: H^1(C_{\rm Nek})/p^n \to H^1_\cF(K,A)$$
and a canonical exact sequence
$$0\to H^1(C_{\rm Nek})/p^n \xrightarrow{q_1} H^1(C_{\rm Nek}/p^n)\to H^2(C_{\rm Nek})[p^n]\to 0.$$
Moreover, we have canonical isomorphisms
$$\kappa_2: H^2(C_{\rm Nek})/p^n \xrightarrow{\sim} H^1_{\cF^\ast}(K,A^\ast(1))^\ast$$
and 
$$q_2: H^2(C_{\rm Nek})/p^n\xrightarrow{\sim} H^2(C_{\rm Nek}/p^n).$$
(See Lemmas \ref{lem Cnek} and \ref{lem NekPT}.)

The first aim of this article is to compare Nekov\'a\v{r}'s Selmer complexes with classical Selmer groups ``on the level of complexes". To do this, we introduce the notion of ``Poitou-Tate complexes". The definition is as follows. In the theory of Mazur-Rubin \cite{MRkoly}, it is important to consider a ``core vertex" $\fn$ and the Poitou-Tate global duality exact sequence
$$0\to H^1_\cF(K,A) \to H^1_{\cF^\fn}(K,A) \xrightarrow{\lambda_\fn} \bigoplus_{v\mid \fn}H^1_{/\cF}(K_v,A)\to H^1_{\cF^\ast}(K,A^\ast(1))^\ast\to 0,$$
where $\cF^\fn$ is the Selmer structure obtained from $\cF$ by ``relaxing" local conditions at $v\mid \fn$ and $\lambda_\fn$ is induced by canonical localization maps. 
We define the Poitou-Tate complex by
$$C_{{\rm PT}}:=\left[H^1_{\cF^\fn}(K,A) \xrightarrow{\lambda_\fn} \bigoplus_{v\mid \fn}H^1_{/\cF}(K_v,A)\right].$$
One checks that, in the derived category, this complex is independent of the choice of a core vertex $\fn$. (See Definition \ref{def PT} and Propositions \ref{prop change S} and \ref{prop core} for details.) By definition, we have
$$H^1(C_{\rm PT})=H^1_\cF(K,A) \text{ and }H^2(C_{\rm PT})=H^1_{\cF^\ast}(K,A^\ast(1))^\ast.$$

We have the following comparison result. 

\begin{theorem}[See Theorem \ref{qithm}]\label{thm1}
Under mild hypotheses, there exists a canonical quasi-isomorphism
$$\varphi: C_{\rm PT} \to C_{\rm Nek}/p^n$$
such that the following diagrams are commutative:
$$\xymatrix{
 H^1(C_{\rm Nek})/p^n \ar[dr]^{q_1} \ar[d]_{\kappa_1}& \\
H^1_\cF(K,A)=H^1(C_{\rm PT}) \ar[r]_-{H^1(\varphi)} & H^1(C_{\rm Nek}/p^n),
} 
$$
$$\xymatrix{
 H^2(C_{\rm Nek})/p^n \ar[rd]^{q_2} \ar[d]_{-\kappa_2}& \\
H^1_{\cF^\ast}(K,A^\ast(1))^\ast = H^2(C_{\rm PT})  \ar[r]_-{H^2(\varphi)}  & H^2(C_{\rm Nek}/p^n) .
}$$
\end{theorem}

\begin{remark}If $T$ and $\sF$ are endowed with the action of a local Gorenstein order $\cR$, then $\varphi$ is $\cR$-equivariant.
\end{remark}


\begin{remark}
Theorem \ref{thm1} was partially proved by Burns and the first author in \cite[Lemma 10.7]{BMC} in the setting of Tate modules of abelian varieties (with the ordinary Greenberg local conditions). Although $\varphi$ specializes to recover the same morphism that was constructed in loc. cit., the latter was not shown to be a quasi-isomorphism, so the claim of Theorem \ref{thm1} is stronger.
\end{remark}

\begin{remark}
A similar result to Theorem \ref{thm1} is proved by Kataoka for the complex $\rgamma(G_{K,S},A)$ (see \cite[Proposition 5.8]{kataoka} and \cite[Proposition 3.23]{ks}). Our method follows \cite{BMC} and is different from \cite{kataoka}. Moreover, Theorem \ref{thm1} is more accurate than Kataoka's result, since we deal with the sign appearing in the second diagram. (This sign is important in the comparison of derived heights: see Theorem \ref{thm4} below.)
\end{remark}


\subsection{Stark systems}\label{s3}

The theory of Stark systems was introduced by Mazur-Rubin \cite{MRselmer} and developed, for example, in \cite{sakamoto}, \cite{sbA}, \cite{bss1}, \cite{bss2}, \cite{kataoka}, \cite{ks}. As an application of Theorem \ref{thm1}, we relate Selmer complexes with Stark systems. 

\subsubsection{The general theory}

For a non-negative integer $r$, let ${\rm SS}_r(A,\cF)$ be the module of Stark systems for $A$ and $\cF$ (see Definition \ref{def stark}). Let $\chi(\cF)$ denote the core rank of $\cF$. (The notion of core rank was introduced by Mazur-Rubin \cite{MRkoly}, but we use a slightly modified definition. See Definition \ref{def core}.) Then it is not difficult to show that there is a canonical isomorphism
$${\det}_{\ZZ/p^n}^{-1}(C_{\rm PT})\simeq {\rm SS}_{\chi(\cF)}(A,\cF).$$
(See Proposition \ref{prop stark}.) Combining this with Theorem \ref{thm1}, we obtain the following. 

\begin{theorem}[See Theorem \ref{thm stark}]\label{thm2}
Under mild hypotheses, there is a canonical isomorphism
$$\varpi_n: {\det}_{\ZZ_p}^{-1}(C_{\rm Nek})/p^n \xrightarrow{\sim} {\rm SS}_{\chi(\cF)}(A,\cF).$$
\end{theorem}

This result is a generalization of Kataoka's result \cite[Theorem 5.6]{kataoka}.

Note that, by the general theory of Stark systems in \cite{bss1}, we know the following: for a basis $\epsilon \in {\rm SS}_{\chi(\cF)}(A,\cF)$ and $i \geq 0$, we can define an ideal $I_i(\epsilon)$ of $\ZZ/p^n$, and we have
\begin{equation}\label{eq fitt}
{\rm Fitt}_{\ZZ/p^n}^i (H^1_{\cF^\ast}(K,A^\ast(1))^\ast) = I_i(\epsilon).
\end{equation}
Here ${\rm Fitt}^i$ denotes the $i$-th Fitting ideal. Using this fact and Theorem \ref{thm2}, we can show that any basis of ${\det}_{\ZZ_p}^{-1}(C_{\rm Nek})$ determines all higher Fitting ideals of $H^1_{\cF^\ast}(K,A^\ast(1))^\ast$. (See Corollary \ref{cor stark}.) 

\begin{remark}\label{equivrk} If $T$ and $\sF$ are endowed with the action of a local Gorenstein order $\cR$, then we actually prove equivariant generalizations of Theorem \ref{thm2} and  (\ref{eq fitt}) pertaining to the determinant ${\det}_{\cR}^{-1}(C_{\rm Nek})$ and the Fitting ideals ${\rm Fitt}_{\cR/p^n}^i (H^1_{\cF^\ast}(K,A^\ast(1))^\ast)$.\end{remark}

\subsubsection{The Heegner point Stark system}

We give an application of Theorem \ref{thm2} to the Heegner point setting. We construct a canonical Stark system by using Heegner points. 

Let $K$ be an imaginary quadratic field. Let $E$ be an elliptic curve defined over $\QQ$. We assume that $E$ has good ordinary reduction at $p$ and every prime divisor of the conductor of $E$ splits in $K$ (Heegner hypothesis). Let $K_\infty/K$ be the anticyclotomic $\ZZ_p$-extension. We set $\Gamma:=\Gal(K_\infty/K)$ and $\Lambda:=\ZZ_p[[\Gamma]]$. 
For a non-negative integer $m$, let $K_m$ be the $m$-th layer of $K_\infty/K$ and set $\Gamma_m:=\Gal(K_m/K)$. 
Let $T$ be the $p$-adic Tate module of $E$ and consider its anticyclotomic deformation $\TT:=T\otimes_{\ZZ_p}\Lambda$. In this setting, we have a natural ``ordinary" Greenberg local condition $\sF$ (see Example \ref{ex greenberg}) and we can consider the corresponding $\Lambda$-adic Selmer complex $\widetilde \rgamma_\sF(K,\TT)$. Under mild hypotheses, the ``Heegner point main conjecture" is proved in \cite{BCK}, and we can construct a canonical $\Lambda$-basis
$$\fz_\infty^{\rm Hg} \in {\det}_{\Lambda}^{-1}(\widetilde \rgamma_\sF(K,\TT))$$
by using Heegner points (see Theorem \ref{HIMC}). Fix $m\geq 0$ and let
$$\fz_m^{\rm Hg} \in {\det}_{\ZZ_p[\Gamma_m]}^{-1}(\widetilde \rgamma_\sF(K,{\rm Ind}_{K_m/K}(T)))$$
be the image of $\fz_\infty^{\rm Hg}$ under the natural surjection
$${\det}_\Lambda^{-1}(\widetilde \rgamma_\sF(K,\TT))\twoheadrightarrow {\det}_{\Lambda}^{-1}(\widetilde \rgamma_\sF(K,\TT))\otimes_\Lambda \ZZ_p[\Gamma_m]\simeq {\det}_{\ZZ_p[\Gamma_m]}^{-1}(\widetilde \rgamma_\sF(K,{\rm Ind}_{K_m/K}(T))). $$
(The last isomorphism is due to the ``control theorem": see \cite[Proposition 8.10.1]{nekovar}.) 
For a positive integer $n$, we set 
$$\cR_{n,m}:=\ZZ/p^n[\Gamma_m] \text{ and }A_{n,m}:={\rm Ind}_{K_m/K}(T/p^n).$$
We then define the Heegner point Stark system for $A_{n,m}$ by
$$\epsilon_{n,m}^{\rm Hg}:=\varpi_{n,m} (\fz_m^{\rm Hg}) \in {\rm SS}_{\chi(\cF)}(A_{n,m},\cF),$$
where $\varpi_{n,m}$ is the canonical isomorphism as in Theorem \ref{thm2} but with $\ZZ_p$ and $T$ replaced by $\ZZ_p[\Gamma_m]$ and ${\rm Ind}_{K_m/K}(T)$ respectively, as allowed by Remark \ref{equivrk} (see also Definition \ref{def heeg stark}). It is worth noting that {\it the core rank $\chi(\cF)$ is zero} in this setting (see Lemma \ref{lem hyp}).

Applying the general result (\ref{eq fitt}) (or rather its generalization alluded to in Remark \ref{equivrk}) to the Heegner point Stark system $\epsilon_{n,m}^{\rm Hg}$, we obtain the following. 

\begin{theorem}[See Theorem \ref{thm heeg}]\label{thm3}
Under mild hypotheses we have
$${\rm Fitt}_{\cR_{n,m}}^i({\rm Sel}_{p^n}(E/K_m)^\ast) = I_i(\epsilon_{n,m}^{\rm Hg})$$
for any $n\geq 1$, $m\geq 0$ and $i\geq 0$. Here ${\rm Sel}_{p^n}(E/K_m)^\ast$ denotes the Pontryagin dual of the $p^n$-Selmer group for $E$ over $K_m$. 
\end{theorem}

This result can be regarded as an equivariant generalization of Kolyvagin's structure theorem \cite{kolyvagin1}, \cite{kolyvagin2} (see also \cite[Theorem 4.7]{zhang}). 

We remark that it is possible to apply Theorem \ref{thm2} to the cyclotomic setting and give a similar result to Theorem \ref{thm3}, which essentially recovers Kurihara's result \cite[Theorem B]{kurihara} for modular symbols. (See Remark \ref{rem kurihara} for details.) Note that our method is totally different: we construct a Stark system directly by using Theorem \ref{thm2}, and we do not consider Kolyvagin systems. It would also be interesting to compare our method with recent works by Kim \cite{kim} and Angurel  \cite{angurel}.

\subsection{Derived $p$-adic heights}\label{s4}

As the second application of Theorem \ref{thm1}, we prove that the derived $p$-adic height pairing introduced by Bertolini-Darmon \cite{BDMT}, \cite{BD der} coincides with the ``higher height pairing" of Nekov\'a\v{r} \cite[\S 11.5]{nekovar}. 

Let $K_\infty/K$ be any $\ZZ_p$-extension. We set $\Gamma:=\Gal(K_\infty/K) $ and $\Lambda:=\ZZ_p[[\Gamma]]$. Let $I\subset \Lambda$ be the augmentation ideal and set
$$Q^k:=I^k/I^{k+1}.$$
Let $E$ be an elliptic curve defined over $K$. We assume that $E$ has good ordinary reduction at every $v\in S_p$. Let ${\rm Sel}_{p^n}(E/K)$ be the $p^n$-Selmer group for $E/K$ and set
$$S_p(E/K):=\varprojlim_n {\rm Sel}_{p^n}(E/K).$$
Under mild hypotheses, Bertolini-Darmon \cite{BD der} defined a filtration
$$S_p(E/K)=S_p^{(1)}\supset S_p^{(2)}\supset \cdots \supset S_p^{(p)}$$
and derived $p$-adic height pairings
\begin{equation*}
\langle \cdot,\cdot \rangle_k^{\rm BD}: S_p^{(k)}\times S_p^{(k)} \to Q^k
\end{equation*}
for $1\leq k\leq p-1$. When $k=1$, this is the usual $p$-adic height pairing. 

On the other hand, Nekov\'a\v{r} \cite[\S 11.5]{nekovar} defined a similar pairing by a systematic method using spectral sequences. Let $T$ be the $p$-adic Tate module of $E$ and set $\TT:=T\otimes_{\ZZ_p}\Lambda$. Let $\widetilde \rgamma_{\sF}(K,\TT)$ be the Selmer complex for $\TT$ attached to the ordinary Greenberg local condition $\sF$. Nekov\'a\v{r} considered the spectral sequence $E_k^{i,j}$ arising from the $I$-adic filtration on $\widetilde \rgamma_{\sF}(K,\TT)$ and defined a pairing
$$\langle \cdot,\cdot \rangle_k^{\rm Nek}: E_k^{0,1}\times E_k^{0,1} \to Q^k.$$

We have the following comparison result. 

\begin{theorem}[See Theorem \ref{main}]\label{thm4}
For any $1\leq k\leq p-1$, we have $E_k^{0,1}=S_p^{(k)}$ and 
$$\langle \cdot,\cdot \rangle_k^{\rm BD} = -\langle \cdot,\cdot\rangle_k^{\rm Nek}.$$
\end{theorem}

We remark that Theorem \ref{thm4} is a refinement of the comparison result of Burns and the first author in \cite[Theorem 10.3]{BMC}, in which only the case $k=1$ is considered. 

In the proof of Theorem \ref{thm4}, Theorem \ref{thm1} (which is a generalization of \cite[Lemma 10.7]{BMC}) plays the key role, since the Poitou-Tate exact sequence is essentially used in the construction of $\langle \cdot,\cdot \rangle_k^{\rm BD}$. 
Another key ingredient is to study relations between ``derived" Bockstein maps (as in \cite{sanoderived}) and ``generalized" Bockstein maps considered by Lam-Liu-Sharifi-Wake-Wan \cite{LLSWW}. We show that these are related in a natural commutative diagram (see Proposition \ref{prop relate}). Using this result, we can prove an algebraic comparison result (see Theorem \ref{compari}). Theorem \ref{thm4} is a direct consequence of Theorems \ref{thm1} and \ref{compari}. 

\subsection{Organization}

The organization of this article is as follows. 

In \S \ref{sec nek selmer}, we review the definition and basic properties of Nekov\'a\v{r}'s Selmer complexes. In \S \ref{sec PT}, we study Poitou-Tate complexes and state the main result (see Theorem \ref{qithm}). \S \ref{sec pf qi} is devoted to the proof of Theorem \ref{qithm}. 

In \S \ref{sec stark}, we review the definition of Stark systems and, as an application of Theorem \ref{qithm}, we prove that the determinant of a Selmer complex is canonically isomorphic to the module of Stark systems (see Theorem \ref{thm stark}). In \S \ref{sec heeg}, we apply Theorem \ref{thm stark} to construct a canonical Heegner point Stark system and prove Theorem \ref{thm3} (= Theorem \ref{thm heeg}). 

In \S \ref{sec statement}, we give a precise statement of Theorem \ref{thm4} (= Theorem \ref{main}). In \S \ref{sec review}, we review the definition of the derived $p$-adic height pairing of Bertolini-Darmon. In \S \ref{sec nek}, we review the definition of the derived $p$-adic height pairing of Nekov\'a\v{r}. Finally, in \S \ref{sec final}, we prove Theorem \ref{main}. 

In appendices, we give necessary algebraic ingredients. In Appendix \ref{sec bock}, we study ``derived" and ``generalized" Bockstein maps and compare them. The results in \S \ref{sec coker} are not used in this article: we give another proof of a result in \cite{sanoderived} (see Corollary \ref{cor sano}). In Appendix \ref{sec abs}, we define two pairings $\langle \cdot,\cdot\rangle_k^{\rm BD}$ and $\langle \cdot,\cdot\rangle_k^{\rm Boc}$ in an abstract setting and prove that they coincide (see Theorem \ref{compari}). This result is used in the proof of Theorem \ref{main}.

\subsection{Notation}


For a commutative ring $R$ and an $R$-module $M$, we set
$$M^\ast:=\Hom_R(M,R). $$
For an element $\pi \in R$, we set
$$M[\pi^n]:=\{a \in M \mid \pi^n \cdot a =0\} \text{ and }M[\pi^\infty]:=\bigcup_n M[\pi^n].$$
The quotient $M/\pi^n M$ is often abbreviated to $M/\pi^n$. 
If $M$ is a $\ZZ_p$-module, then we set
$$M^\vee :=\Hom_{\ZZ_p}(M,\QQ_p/\ZZ_p). $$

For a profinite group $G$, let $\rgamma(G,-):=C^\bullet_{\rm cont}(G,-)$ be the complex of continuous cochains as in \cite[\S 3.4]{nekovar}. The cohomology groups of $\rgamma(G,-)$ are denoted by $H^i(G,-)$.  For a field $F$, the absolute Galois group of $F$ is denoted by $G_F$. 
We abbreviate $\rgamma(G_F,-)$ and $H^i(G_F,-)$ to $\rgamma(F,-)$ and $H^i(F,-)$ respectively.
For a Galois extension $L/F$, we 
sometimes abbreviate $\rgamma(\Gal(L/F),-)$ and $H^i(\Gal(L/F),-)$ to $\rgamma(L/F,-)$ and $H^i(L/F,-)$ respectively.

Let $K$ be a number field. The set of infinite places of $K$ is denoted by $S_\infty(K)$. The set of places of $K$ lying above a given prime number $p$ is denoted by $S_p(K)$. We often abbreviate $S_\infty(K)$ and $S_p(K)$ to $S_\infty$ and $S_p$ respectively. 

For a finite set $S$ of places of $K$ containing $S_\infty$, let $K_S/K$ be the maximal extension unramified outside $K$ and set $G_{K,S}:=\Gal(K_S/K)$. We set
$$S_f:=S\setminus S_\infty.$$

For a $G_K$-module $X$ and a finite abelian extension $L/K$, let ${\rm Ind}_{L/K}(X):={\rm Ind}_{G_L}^{G_K}(X)$ be the induced module. 

\begin{acknowledgments} 
We would like to thank Dominik Bullach for carefully reading an earlier draft of this article and giving helpful comments. We would also like to thank Alberto Angurel Andres, Takenori Kataoka, Chan-Ho Kim, and Masato Kurihara for helpful discussions and comments. 

The first author acknowledges support for this article as part of Grants CEX2023-001347-S, PID2022-142024NB-I00 and CNS2023-145167 funded by \newline MICIU/AEI/10.13039/501100011033. 

The second author was supported by JSPS KAKENHI Grant Number 22K13896.
\end{acknowledgments}

\section{Selmer complexes}

\subsection{Nekov\'a\v{r}'s Selmer complexes}\label{sec nek selmer}

We give a review on Nekov\'a\v{r}'s Selmer complexes \cite{nekovar}. 

Let $p$ be an odd prime number. Let $K$ be a number field. We fix a finite set $S$ of places of $K$ which contains $S_\infty\cup S_p$. Let $\Phi/\QQ_p$ be a finite extension. Let $\cO:=\cO_\Phi$ be its ring of integers and $\pi\in \cO$ a prime element. Let $\cR$ be a local Gorenstein $\cO$-order. Let $T$ be a free $\cR$-module of finite rank endowed with a continuous linear action of $G_{K,S}$. 

Suppose that the following data are given: for each $v\in S_p$, an exact sequence of $\cR[G_{K_v}]$-modules
\begin{equation}\label{greenberg local}
\sF_v: 0\to T_v^+\to T\to T_v^-\to 0
\end{equation}
such that $T_v^\pm$ are free as $\cR$-modules. 

We also assume that $T^{I_v}$ is a free $\cR$-module for every  $v \in S_f\setminus S_p$, where $I_v$ denotes the inertia subgroup of $G_{K_v}$. 

\begin{example}\label{ex greenberg}
Let $E$ be an elliptic curve defined over $K$. Assume that $E$ has good ordinary reduction at every prime lying above $p$. For $v\in S_p$, let $\widetilde E_v$ be the reduction of $E$ modulo $v$. Let $T:=T_p(E)$ be the $p$-adic Tate module of $E$. Then, by setting $T_v^-:=T_p(\widetilde E_v)$ and defining $T_v^+$ as the kernel of the reduction map $T\to T_v^-$, we obtain an exact sequence of the form (\ref{greenberg local}). 

More generally, we can consider the following equivariant setting. Let $L/K$ be a finite abelian $p$-extension unramified outside $p$ and set $G:=\Gal(L/K)$. Then $\cR:=\ZZ_p[G]$ is a local Gorenstein $\ZZ_p$-order. If we set $T:={\rm Ind}_{L/K}(T_p(E))$, then it has an exact sequence of the form (\ref{greenberg local}) for each $v\in S_p$ by setting $T_v^-:={\rm Ind}_{L/K}(T_p(\widetilde E_v))$ and $T_v^+:=\ker(T\to T_v^-)$. 
\end{example}

For a finite place $v$ of $K$, the unramified cohomology complex is defined by
$$\rgamma_{\rm ur}(K_v, T):= \rgamma(K_v^{\rm ur}/K_v, T^{I_v}),$$
where $K_v^{\rm ur}$ denotes the maximal unramified extension of $K_v$ and $I_v:=G_{K_v^{\rm ur}}$ the inertia group. We set
$$\rgamma_{/{\rm ur}}(K_v,T):={\rm Cone}\left(\rgamma_{\rm ur}(K_v,T) \xrightarrow{-{\rm Inf}_{K_v^{\rm ur}/K_v}} \rgamma(K_v,T)\right),$$
where ${\rm Inf}_{K_v^{\rm ur}/K_v}$ is the inflation morphism. 

\begin{definition}\label{def sel}
Let $\sF=(\sF_v)_{v\in S_p}$ be the collection of exact sequences as in (\ref{greenberg local}). We define {\it Nekov\'a\v{r}'s Selmer complex} attached to $\sF$ by
$$\widetilde \rgamma_\sF(K,T) := {\rm Cone}\left(\rgamma(G_{K,S},T)\to \bigoplus_{v\in S_p}\rgamma(K_v,T_v^-)\oplus \bigoplus_{v\in S_f\setminus S_p}\rgamma_{/{\rm ur}}(K_v,T)\right)[-1].$$
\end{definition}

\begin{remark}
The complex $\widetilde \rgamma_\sF(K,T)$ is independent of $S$ up to quasi-isomorphism (see \cite[Prop. 7.8.8(ii)]{nekovar}). 
\end{remark}

\begin{proposition}\label{prop perf}\
\begin{itemize}
\item[(i)] $\widetilde \rgamma_\sF(K,T)$ is a perfect complex of $\cR$-modules. 
\item[(ii)] We set
$$\chi(\sF):= \sum_{v\in S_\infty}{\rm rank}_{\cR}(H^0(K_v,T^\ast(1))) - \sum_{v\in S_p}[K_v:\QQ_p] \cdot {\rm rank}_{\cR}(T_v^-).$$
Then the Euler characteristic of $\widetilde \rgamma_\sF(K,T)$ is $-\chi(\sF)$, i.e., for any bounded complex $C^\bullet$ of free $\cR$-modules which is quasi-isomorphic to $\widetilde \rgamma_\sF(K,T)$, we have
$$\sum_{i\in \ZZ} (-1)^i {\rm rank}_\cR(C^i)= -\chi(\sF). $$
\end{itemize}
\end{proposition}

\begin{proof}
Claim (i) follows from \cite[Proposition 9.7.2(ii)]{nekovar}. (Note that $\rgamma_{\rm ur}(K_v,T)$ is perfect for $v\in S_f\setminus S_p$ since we assume $T^{I_v}$ is a free $\cR$-module.) Claim (ii) follows from \cite[Theorem 7.8.6]{nekovar}.
\end{proof}

We set
$$V:=\Phi\otimes_\cO T \text{ and }A:=V/T.$$
Then the exact sequence (\ref{greenberg local}) naturally gives the exact sequences
$$0\to V_v^+\to V\to V_v^-\to 0 \text{ and }0\to A_v^+\to A\to A_v^-\to 0.$$
Using these, we can define Selmer complexes for $V$ and $A$
$$\widetilde \rgamma_\sF(K,V)\text{ and }\widetilde \rgamma_\sF(K,A)$$
exactly in the same way.

\subsection{Poitou-Tate complexes}\label{sec PT}

We continue to suppose that the data $\sF=(\sF_v)_{v\in S_p}$ as in (\ref{greenberg local}) are given. 

Let $n$ be a positive integer. We set
$$\cR_n:=\cR/\pi^n \cR \text{ and }A_n:=T/\pi^n T.$$
Note that $\sF_v$ induces an exact sequence
$$0\to A_{n,v}^+\to A_n \to A_{n,v}^-\to 0$$
for each $v\in S_p$. 

For a finite place $v$, let $H^1_f(K_v,-)$ denote the Bloch-Kato local condition (as in \cite[\S 1.3]{R}). For the definition of a general Selmer structure, see \cite[Definition 2.1.1]{MRkoly}. We consider the following Selmer structure attached to $\sF$. 

\begin{definition}\label{def sel str}
Let $X \in \{T,V,A,A_n\}$. 
We define a Selmer structure $\cF$ on $X$ (attached to $\sF$) as follows. 
\begin{itemize}
\item If $v\in S_p$, then we set
$$H^1_\cF(K_v,X):= \im (H^1(K_v,X_v^+)\to H^1(K_v,X)).$$
\item If $v \notin S_p$, then we set 
$$H^1_\cF(K_v,X):= H^1_f(K_v,X).$$
\end{itemize}
\end{definition}

The Selmer group attached to the Selmer structure $\cF$ is defined by
$$H^1_\cF(K,X):=\ker\left(H^1(G_{K,S},X)\to \bigoplus_{v\in S_f}H^1_{/\cF}(K_v,X)\right),$$
where for each $v$, $H^1_{/\cF}(K_v,X)$ denotes the quotient of $H^1(K_v,X)$ by $H^1_{\cF}(K_v,X)$.

Let $\cF^\ast$ be the dual of the Selmer structure $\cF$ (see \cite[\S 2.3]{MRkoly}). 
Let 
\begin{equation}\label{PT}
0\to H^1_\cF(K,A_n) \to H^1(G_{K,S},A_n)\xrightarrow{\lambda} \bigoplus_{v\in S_f}H^1_{/\cF}(K_v,A_n)\xrightarrow{\nu} H^1_{\cF^\ast}(K,A_n^\ast(1))^\ast
\end{equation}
be the Poitou-Tate global duality exact sequence (see \cite[Theorem 2.3.4]{MRkoly}). We mainly consider the case when the last map $\nu$ is surjective:

\begin{hypothesis}\label{hyp surj}
The map $\nu$ in (\ref{PT}) is surjective. 
\end{hypothesis}

Using the Poitou-Tate sequence (\ref{PT}), we give the following definition. 

\begin{definition}\label{def PT}
We define a {\it Poitou-Tate complex} by
$$C_{{\rm PT},S}=C_{{\rm PT},S,n} := \left[ H^1(G_{K,S},A_n)\xrightarrow{\lambda} \bigoplus_{v\in S_f}H^1_{/\cF}(K_v,A_n)\right],$$
where the first term is placed in degree one. 
\end{definition}

\begin{remark}
By definition, note that 
$$H^1(C_{{\rm PT},S})= H^1_\cF(K,A_n)$$
and, if Hypothesis \ref{hyp surj} is satisfied, then
$$H^2(C_{{\rm PT},S}) = H^1_{\cF^\ast}(K,A_n^\ast(1))^\ast. $$
\end{remark}

\begin{proposition}\label{prop change S}
Let $S'$ be a finite set of places of $K$ such that $S\subset S'$. Assume that Hypothesis \ref{hyp surj} is satisfied for $S$. Then Hypothesis \ref{hyp surj} is also satisfied for $S'$, and $C_{{\rm PT},S}$ is canonically quasi-isomorphic to $C_{{\rm PT},S'}$. 
\end{proposition}

\begin{proof}
Note that $K_S \subset K_{S'}$ if $S\subset S'$. The inflation map
$$H^1(G_{K,S},A_n)\to H^1(G_{K,S'},A_n)$$
(which is injective) and the natural injection
$$\bigoplus_{v\in S_f}H^1_{/\cF}(K_v,A_n) \hookrightarrow \bigoplus_{v\in S_f'} H^1_{/\cF}(K_v,A_n)$$
define a morphism of complexes
$$C_{{\rm PT},S}\to C_{{\rm PT},S'}.$$
If Hypothesis \ref{hyp surj} is satisfied for $S$, then it is satisfied also for $S'$ by \cite[Theorem 2.3.4]{MRkoly}. The morphism defined above is clearly a quasi-isomorphism. 
\end{proof}

\begin{remark}\label{rem abb}
In the following, we often abbreviate $C_{{\rm PT},S}$ to $C_{\rm PT}$. This is justified by Proposition \ref{prop change S}.
\end{remark}

We now assume the following.

\begin{hypothesis}\label{hyp tech}\
\begin{itemize}
\item[(i)] $H^0(K,A)=H^0(K,T^\vee(1))=0$.
\item[(ii)] $H^0(K_v,T_v^-)=0$ for any $v\in S_p$. 
\item[(iii)] $H^0(K_v,A_v^-)=H^0(K_v,(T_v^+)^\vee(1))=0$ for any $v\in S_p$. 
\item[(iv)] $H^1_{\rm ur}(K_v,T)=H^1_f(K_v,T)$ for any finite place $v$ with $v\notin S_p$. 
\end{itemize}
\end{hypothesis}

\begin{remark}\label{HypimplyHyp}
In the setting of Example \ref{ex greenberg}, Hypothesis \ref{hyp tech} is satisfied if $p$ does not divide $\# E(K)_{\rm tors}\cdot {\rm Tam}(E/K)\cdot \prod_{v\mid p} \# \widetilde E_v(\FF_v)$, 
where ${\rm Tam}(E/K)$ denotes the product of Tamagawa factors of $E/K$ and $\FF_v$ the residue field of $v$. 
\end{remark}

\begin{lemma}\label{lem cartesian}
Assume Hypothesis \ref{hyp tech}(iii). Then, for any $v\in S_p$, the natural surjection $T\twoheadrightarrow A_n$ induces a surjection
$$H^1_\cF(K_v, T) \twoheadrightarrow H^1_\cF(K_v,A_n).$$
\end{lemma}

\begin{proof}
We have $H^2(K_v,T_v^+)\simeq H^0(K_v,(T_v^+)^\vee(1))^\vee =0$ by Hypothesis \ref{hyp tech}(iii). This implies that the map
$$H^1(K_v,T_v^+) \to H^1(K_v,A_{n,v}^+)$$
is surjective. The claim follows from the definition of the Selmer structure $\cF$. 
\end{proof}

\begin{lemma}\label{lem sel}
Assume Hypothesis \ref{hyp tech}. 
\begin{itemize}
\item[(i)] We have a natural identification $\widetilde H^1_\sF(K,T)=H^1_\cF(K,T)$. 
\item[(ii)] There is a canonical isomorphism $\widetilde H^2_{\sF}(K,T) \simeq  H^1_{\cF^\ast}(K,T^\vee(1))^\vee$. 
\item[(iii)] The natural injection $A_n^\ast(1)\hookrightarrow T^\vee(1)$ (which is the dual of the natural surjection $T \twoheadrightarrow T/\pi^nT=A_n$) induces an isomorphism
$$H^1_{\cF^\ast}(K,A_n^\ast(1))\xrightarrow{\sim} H^1_{\cF^\ast}(K,T^\vee(1))[\pi^n].$$
\end{itemize}
\end{lemma}

\begin{proof}
(i) By definition, $\widetilde H^1_\sF(K,T)$ lies in the exact sequence
$$\bigoplus_{v\in S_p}H^0(K_v,T_v^-) \to \widetilde H^1_\sF(K,T)\to H^1(G_{K,S},T) \to \bigoplus_{v\in S_p} H^1(K_v,T_v^-)\oplus \bigoplus_{v\in S_f\setminus S_p}H^1_{/{\rm ur}}(K_v,T).$$
By Hypothesis \ref{hyp tech}(ii) and (iv), and the definition of $\cF$, we have
$$\widetilde H^1_\sF(K,T)=\ker \left( H^1(G_{K,S},T) \to \bigoplus_{v\in S_p}H^1(K_v,T_v^-)\oplus \bigoplus_{v\in S_f\setminus S_p}H^1_{/f}(K_v,T)\right) = H^1_\cF(K,T). $$
This proves claim (i). 

(ii) Let $\widetilde \rgamma_{\sF^\ast}(K,T^\vee(1))$ be the Selmer complex defined by the dual of (\ref{greenberg local}):
$$\sF_v^\ast: 0 \to (T_v^-)^\vee(1) \to T^\vee(1)\to (T_v^+)^\vee(1)\to 0$$
for each $v\in S_p$. Then, by duality (see \cite[Proposition 9.7.2(i)]{nekovar}), we have a canonical isomorphism
$$\widetilde H^2_\sF(K,T)\simeq \widetilde H^1_{\sF^\ast}(K,T^\vee(1))^\vee.$$
Hence it is sufficient to show $\widetilde H^1_{\sF^\ast}(K,T^\vee(1))\simeq H^1_{\cF^\ast}(K,T^\vee(1))$. By definition, $\widetilde H^1_{\sF^\ast}(K,T^\vee(1))$ lies in the exact sequence
\begin{multline*}
\bigoplus_{v\in S_p}H^0(K_v,(T_v^+)^\vee(1)) \to \widetilde H^1_{\sF^\ast}(K,T^\vee(1))\to H^1(G_{K,S},T^\vee(1)) \\
\to \bigoplus_{v\in S_p} H^1(K_v,(T_v^+)^\vee(1))\oplus \bigoplus_{v\in S_f\setminus S_p}H^1_{/{\rm ur}}(K_v,T^\vee(1)).
\end{multline*}
The claim follows from this because the first term in this exact sequence vanishes by Hypothesis \ref{hyp tech}(iii) and because for each $v\in S_f\setminus S_p$,
$$H^1_{/{\rm ur}}(K_v,T^\vee(1))\simeq H^1_{{\rm ur}}(K_v,T)^\vee=H^1_{f}(K_v,T)^\vee=H^1_{\cF}(K_v,T)^\vee\simeq H^1_{\cF^\ast}(K_v,T^\vee(1))$$ by Hypothesis \ref{hyp tech}(iv).

(iii) We set $X:=T^\vee(1)$. Then note that $X[\pi^n]=A_n^\ast(1)$. Since we have $H^0(K,X)=0$ by Hypothesis \ref{hyp tech}(i), the natural injection $X[\pi^n]\hookrightarrow X$ induces an isomorphism
$$H^1(G_{K,S},X[\pi^n])\xrightarrow{\sim} H^1(G_{K,S},X)[\pi^n].$$
Consider the following commutative diagram with exact rows:
$$
\xymatrix{
0 \ar[r] & H^1_{\cF^\ast}(K,X[\pi^n]) \ar[r] \ar[d]& H^1(G_{K,S},X[\pi^n]) \ar[r] \ar[d]^\simeq& \bigoplus_{v\in S_f}H^1_{/\cF^\ast}(K_v,X[\pi^n])\ar[d]\\
0 \ar[r] & H^1_{\cF^\ast}(K,X)[\pi^n] \ar[r]& H^1(G_{K,S},X)[\pi^n] \ar[r]& \bigoplus_{v\in S_f}H^1_{/\cF^\ast}(K_v,X)[\pi^n].\\
}
$$
To prove the claim, it is sufficient to check that the right vertical arrow is an injection, or equivalently, the natural map
$$H^1_\cF(K_v,T) \to H^1_\cF(K_v,A_n)$$
is surjective for each $v\in S_f$. If $v\in S_f\setminus S_p$, then this follows from the fact that $H^1_f(K_v,A_n)$ is the image of $H^1_f(K_v,T)$ (see \cite[Lemma 1.3.8(i)]{R}). If $v\in S_p$, the claim is proved in Lemma \ref{lem cartesian}. This completes the proof of claim (iii). 
\end{proof}

We set 
$$C_{\rm Nek}:=\widetilde \rgamma_\sF(K,T).$$
In the rest of this section we fix $n\geq 1$ and, after recalling that $\cR_n:=\cR/\pi^n \cR$, we also set
$$C_{\rm Nek}/\pi^n:= C_{\rm Nek}\lotimes_\cR \cR_n.$$ 

\begin{lemma}\label{lem Cnek}
Assume Hypothesis \ref{hyp tech}(i). Then $C_{\rm Nek}$ and $C_{\rm Nek}/\pi^n$ are acyclic outside degrees one and two, and the following hold.  
\begin{itemize}
\item[(i)] There is a canonical exact sequence
$$0\to H^1(C_{\rm Nek})/\pi^n\xrightarrow{q_1} H^1(C_{\rm Nek}/\pi^n) \to H^2(C_{\rm Nek})[\pi^n]\to 0.$$
\item[(ii)] There is a canonical isomorphism
$$q_2:H^2(C_{\rm Nek})/\pi^n\xrightarrow{\sim} H^2(C_{\rm Nek}/\pi^n).$$
\end{itemize}
\end{lemma}

\begin{proof}
By $H^0(K,T^\vee(1))=0$, we have $H^3(C_{\rm Nek})=0$ by duality (see \cite[Proposition 9.7.2(iv)]{nekovar}). Also, by $H^0(K,A)=0$, we have $H^0(K,T)=0$ and $H^0(C_{\rm Nek})=0$. Note that the same assumption also implies that $H^1(C_{\rm Nek})$ is $\cO$-free (see also Lemma \ref{tflemma}(iii) below). The claim follows immediately from the long exact sequence associated to the exact triangle
$$C_{\rm Nek}\xrightarrow{\pi^n} C_{\rm Nek}\to C_{\rm Nek}/\pi^n.$$
\end{proof}

Let $\fn$ be a square-free product of finite places of $K$. 
Let $\cF^\fn$ be the Selmer structure obtained from $\cF$ by relaxing the conditions at $v \mid \fn$, namely, 
$$H^1_{\cF^\fn}(K_v,-):=\begin{cases}
H^1(K_v,-) &\text{if $v\mid \fn$,}\\
H^1_\cF(K_v,-) &\text{if $v\nmid \fn$}.
\end{cases}$$

Recall (from \cite[Definition 3.12]{bss2}, for example) the definition of ``core vertices". 

\begin{definition}\label{def core}
Let $\fn$ be a square-free product of places of $K$ which do not belong to $S_\infty\cup S_p\cup S_{\rm ram}(T)$, where $S_{\rm ram}(T)$ denotes the set of places of $K$ at which $T$ ramifies. We say that $\fn$ is a {\it core vertex} for $\cF$ on $A_n$ if 
\begin{itemize}
\item $H^1_{\cF^\fn}(K,A_n)$ is a free $\cR_n$-module, 
\item $H^1_{/\cF}(K_v,A_n)$ is a free $\cR_n$-module for every $v\mid \fn$, and
\item $H^1_{(\cF^\fn)^\ast}(K,A_n^\ast(1))=0$. 
\end{itemize}
(Note that our definition of core vertices is slightly general, since we do not assume $H^1_{/\cF}(K_v,A_n)$ is free {\it of rank one} for $v\mid \fn$.) 
We define the {\it core rank} $\chi(\cF)$ by
\begin{equation}\label{core formula}
\chi(\cF):={\rm rank}_{\cR_n}(H^1_{\cF^\fn}(K,A_n)) - \sum_{v\mid \fn} {\rm rank}_{\cR_n}(H^1_{/\cF}(K_v,A_n)).
\end{equation}
One checks that this is well-defined, namely, independent of the choice of a core vertex. (In fact, $-\chi(\cF)$ coincides with the Euler characteristic of $C_{{\rm PT}}$: see Proposition \ref{prop core}(ii) below.)
\end{definition}

We consider the following condition. 

\begin{hypothesis}\label{hyp core}
There exists a core vertex for $\cF$ on $A_n$ with core rank
$$\chi(\cF)= \chi(\sF),$$
where $\chi(\sF)$ is defined in Proposition \ref{prop perf}(ii). 
\end{hypothesis}

\begin{proposition}\label{prop core}
Assume Hypothesis \ref{hyp core} and let $\fn$ be a core vertex. 
\begin{itemize}
\item[(i)] We have the Poitou-Tate exact sequence
$$0\to H^1_\cF(K,A_n) \to H^1_{\cF^\fn}(K,A_n) \xrightarrow{\lambda_\fn} \bigoplus_{v\mid \fn}H^1_{/\cF}(K_v,A_n)\to H^1_{\cF^\ast}(K,A_n^\ast(1))^\ast\to 0.$$
\item[(ii)] If $S$ contains all prime divisors of $\fn$, then Hypothesis \ref{hyp surj} is satisfied for $S$, and $C_{{\rm PT},S}$ is canonically quasi-isomorphic to the complex
$$C_{{\rm PT},\fn}:=\left[H^1_{\cF^\fn}(K,A_n) \xrightarrow{\lambda_\fn} \bigoplus_{v\mid \fn}H^1_{/\cF}(K_v,A_n)\right].$$
In particular, $C_{{\rm PT},S}$ is a perfect complex of $\cR_n$-modules. 
\end{itemize}
\end{proposition}

\begin{proof}
Claim (i) follows from \cite[Theorem 2.3.4]{MRkoly} and $H^1_{(\cF^\fn)^\ast}(K,A_n^\ast(1))=0$. Claim (ii) follows by the same argument as Proposition \ref{prop change S}. 
\end{proof}

\begin{lemma}\label{lem NekPT}
Assume Hypothesis \ref{hyp tech}. 
\begin{itemize}
\item[(i)] There is a canonical map
$$\kappa_1: H^1(C_{\rm Nek})/\pi^n \to H^1_\cF(K,A_n).$$
\item[(ii)] There is a canonical isomorphism
$$\kappa_2: H^2(C_{\rm Nek})/\pi^n \xrightarrow{\sim} H^1_{\cF^\ast}(K,A_n^\ast(1))^\ast.$$
\item[(iii)] If we further assume Hypothesis \ref{hyp core}, then we have
$$\# H^1(C_{\rm Nek}/\pi^n) = \# H^1_\cF(K,A_n).$$
\end{itemize}
\end{lemma}

\begin{proof}
(i) By Lemma \ref{lem sel}(i), we have
$$H^1(C_{\rm Nek})= \widetilde H^1_\sF(K,T) = H^1_\cF(K,T). $$
The map $\kappa_1$ is then defined to be the map $H^1_\cF(K,T)/\pi^n\to H^1_\cF(K,A_n)$ induced by the natural surjection $T\twoheadrightarrow T/\pi^n T=A_n$. 

(ii) By Lemma \ref{lem sel}(ii), we have
$$H^2(C_{\rm Nek})/\pi^n = \widetilde H^2_\sF(K,T)/\pi^n \simeq H^1_{\cF^\ast}(K,T^\vee(1))[\pi^n]^\vee.$$
The claim then follows from Lemma \ref{lem sel}(iii). 

(iii) Assume Hypothesis \ref{hyp core} and let $\fn$ be a core vertex. Then, by Proposition \ref{prop core} and (\ref{core formula}), we see that the Euler characteristic of $C_{{\rm PT},\fn}$ is $-\chi(\sF)$. Therefore, by Proposition \ref{prop perf}, it is the same as the Euler characteristic of $C_{\rm Nek}/\pi^n$, and so
$$\frac{\# H^1(C_{\rm Nek}/\pi^n)}{\# H^2(C_{\rm Nek}/\pi^n)} = \frac{\# H^1_\cF(K,A_n)}{\# H^1_{\cF^\ast}(K,A_n^\ast(1))^\ast}.$$
Since we have $\# H^2(C_{\rm Nek}/\pi^n) = \# H^1_{\cF^\ast}(K,A_n^\ast(1))^\ast$ by (ii) and Lemma \ref{lem Cnek}(ii), the claim follows. 
\end{proof}

\begin{theorem}\label{qithm}
Assume Hypotheses \ref{hyp tech} and \ref{hyp core}. 
Let $C_{\rm PT}$ be the Poitou-Tate complex in Definition \ref{def PT} (see also Remark \ref{rem abb}). 
Then there exists a canonical quasi-isomorphism
$$\varphi: C_{\rm PT} \to C_{\rm Nek}/\pi^n$$
such that the following diagrams are commutative:
$$\xymatrix{
 H^1(C_{\rm Nek})/\pi^n \ar[dr]^{q_1} \ar[d]_{\kappa_1}& \\
H^1_\cF(K,A_n)=H^1(C_{\rm PT}) \ar[r]_-{H^1(\varphi)} & H^1(C_{\rm Nek}/\pi^n),
} 
$$
$$\xymatrix{
 H^2(C_{\rm Nek})/\pi^n \ar[rd]^{q_2} \ar[d]_{-\kappa_2}& \\
H^1_{\cF^\ast}(K,A_n^\ast(1))^\ast = H^2(C_{\rm PT})  \ar[r]_-{H^2(\varphi)}  & H^2(C_{\rm Nek}/\pi^n) ,
}$$
where the maps $\kappa_i$ and $q_i$ are as in Lemmas \ref{lem NekPT} and \ref{lem Cnek} respectively. 
\end{theorem}

\subsection{Proof of Theorem \ref{qithm}}\label{sec pf qi}

In this subsection, we give a proof of Theorem \ref{qithm}. 

We first note the following. Since we assume Hypothesis \ref{hyp core}, there exists a core vertex $\fn$. Note that, by Proposition \ref{prop change S}, we may assume that $S$ is sufficiently large so that it contains all prime divisors of $\fn$. Then, in particular, Hypothesis \ref{hyp surj} is satisfied by Proposition \ref{prop core}(ii). 

The proof proceeds as follows. In \S \ref{step1}, we give some preliminaries. In \S \ref{step2}, we give an explicit representative of Nekov\'a\v{r}'s Selmer complex $C_{\rm Nek}$. In \S \ref{step3}, we give a definition of the morphism $\varphi$. In \S \ref{step4}, we prove the commutativity of the diagrams in Theorem \ref{qithm}. Finally, in \S \ref{step5}, we prove that $\varphi$ is a quasi-isomorphism and complete the proof. 

\subsubsection{Preliminaries}\label{step1}

We begin with the following observation. 

\begin{lemma}\label{tflemma} Assume Hypothesis \ref{hyp tech}. Then the following claims are valid.
\begin{itemize}
    \item[(i)] For each $v\in S_p$, the $\cO$-module $H^1(K_v,T_v^-)$ is torsion-free.
    \item[(ii)] For each $v\in S_f\setminus S_p$, the $\cO$-module $H^1_{/{\rm ur}}(K_v,T)$ is torsion-free.
    \item[(iii)] The $\cO$-module $H^1(G_{K,S},T)$ is torsion-free.
\end{itemize}
\end{lemma}
\begin{proof}
To prove claim (i) it suffices to fix $v\in S_p$ and prove that the action of $\pi$ defines an injective endomorphism of $H^1(K_v,T_v^-)$. The short exact sequence $0\to T_v^-\xrightarrow{\pi}T_v^-\to A_{1,v}^-\to 0$ induces an exact sequence $H^0(K_v,A_{1,v}^-)\to H^1(K_v,T_v^-)\xrightarrow{\pi}H^1(K_v,T_v^-)$. But the map $\cO/\pi\cO\to\Phi/\cO$ induced by multiplication by $\pi^{-1}$ is injective and, since $T_v^-$ is $\cO$-flat, induces an injective map
$$A_{1,v}^-\cong(\cO/\pi\cO)\otimes_\cO T_v^-\longrightarrow(\Phi/\cO)\otimes_\cO T_v^-\cong A_v^-.$$
In turn, this induces an injective map $H^0(K_v,A_{1,v}^-)\to H^0(K_v,A_{v}^-)$. Since we have $H^0(K_v,A_v^-)=0$ by Hypothesis \ref{hyp tech}(iii),  this completes the proof of claim (i).

Given $v\in S_f\setminus S_p$, claim (ii) follows from 
the fact that $H^1(K_v,T)/H^1_f(K_v,T)$ is $\cO$-torsion-free {\color{black}and Hypothesis \ref{hyp tech}(iv)}.

{\color{black}
Claim (iii) follows from $H^0(K,A)=0$ (by Hypothesis \ref{hyp tech}(i)) and the same argument that was used to prove claim (i).
}
\end{proof}

For each place $v$ of $K$ we abbreviate $G_{K_v}$ to $G_v$. In the next intermediate result we let $$(G,X,X_n)\in\{(G_{K,S},T,A_n)\}\cup\{(G_v,T_v^-,A_{n,v}^-)\}_{v\in S_p}\cup\{(G_v,T,A_{n})\}_{v\in S_f\setminus S_p}.$$
For $Y\in\{X,X_n\}$ we then define $d^0:\rgamma^0(G,Y)=Y\to \rgamma^1(G,Y)$ by $(d^0(y))(g)=g\cdot y-y$ for each $y\in Y$, $g\in G$, and $d^1:\rgamma^1(G,Y)\to \rgamma^2(G,Y)$ by $$(d^1(\rho))(g,h)=g\cdot\rho(h)-\rho(gh)+\rho(g)$$
for each $\rho\in\rgamma^1(G,Y)$ and $g,h\in G$.

\begin{lemma}\label{lifts} There is an isomorphism $l$ of $\cR$-modules that makes the triangle
\begin{equation*}
\xymatrix{& \rgamma^1(G,X) \ar@{->}[dl]_{\mu}  \ar@{->}[dr]^{\xi} \\ \rgamma^1(G,X_n)  \ar@{->}[rr]^{l} & &(\cO/\pi^n\cO)\otimes_{\cO}\rgamma^{1}(G,X)  }
\end{equation*}
commute, where $\mu$ and $\xi$ are the canonical maps. In addition, $l$ restricts to give homomorphisms
$$\ker(d^1)\to\ker((\cO/\pi^n\cO)\otimes d^1)\,\,\,\,\,\,\text{ and }\,\,\,\,\,\,\im(d^0)\to\im((\cO/\pi^n\cO)\otimes d^0).$$
\end{lemma}
\begin{proof}To first prove that there is a homomorphism $l$ that makes the given triangle commute we fix $\rho\in \rgamma^1(G,X_n)$ and claim that the element $1\otimes\rho'$ of $(\cO/\pi^n\cO)\otimes_{\cO}\rgamma^{1}(G,X)$ is independent of the choice of continuous lift $\rho'$ of $\rho$ through the canonical projection $\varpi:X\to X_n$ (such lifts exist since $X_n$ is a discrete space). 

We let $\rho'$ and $\rho''$ denote continuous lifts of $\rho$ through $\varpi$.
Since $\varpi\circ\rho'=\rho=\varpi\circ\rho''$ we have $\varpi\circ(\rho'-\rho'')=0$. It follows that $$(\rho'-\rho'')(g)\in\ker(\varpi)=\pi^n\cdot X$$ for all $g\in G$ and, since $X$ is $\cO$-free, there exists a unique $x_g\in X$ for which $$(\rho'-\rho'')(g)= \pi^nx_g.$$ We then obtain a continuous function $x:G\to X$ by setting $x(g):=x_g$. Then \begin{equation}\label{previousparagraph}1\otimes(\rho'-\rho'')=1\otimes \pi^nx=\pi^n\otimes x=0\otimes x,\end{equation} so $1\otimes\rho'=1\otimes\rho''$. {\color{black}Hence we have proved that the homomorphism $l$ given by $l(\rho):=1\otimes \rho'$ is well-defined.}

Next we observe that both $\mu$ (again, because $X_n$ is a discrete space) and $\xi$ are surjective. Thus, the kernel-cokernel exact sequence of the given commutative triangle is of the form
$$0\to\ker(\mu)\to\ker(\xi)\xrightarrow{\mu}\ker(l)\to0$$
and, in particular, $l$ is surjective. In addition,
$$\ker(l)=\mu(\ker(\xi))=\mu(\pi^n\cdot \rgamma^1(G,X))=\pi^n\cdot\im(\mu)=0.$$
This shows that $l$ is an isomorphism. The additional claims are straightforward to verify explicitly.
\end{proof}

\subsubsection{An explicit representative}\label{step2}

In order to explicitly define the morphism $\varphi$ in Theorem \ref{qithm} we first describe an explicit representative of $C_{\rm Nek}$. The claimed result is clearly independent of such a choice.

For any $v\in S_p$ we set
$$C_v^2:=\rgamma^1(G_v,T_v^-)/\im(d^0_v)$$ and also, for any $i\geq 3$, $C_v^{i}:=\rgamma^{i-1}(K_v,T_v^-)$. Here $d^0_v$ is the map $d_0$ for $(G,Y)=(G_v,T_v^-)$.

For any $v\in S_f\setminus S_p$ we set
$$C_v^2:=\frac{\left(\rgamma^{1}(G_v,T)/\im(d^0_v)\right)}{{\color{black}H^1_{\rm ur}(K_v,T)}}$$
and also, for any $i\geq 3$, $C_v^i:=\rgamma^{i-1}(G_v,T)$. Here $d^0_v$ is the map $d_0$ for $(G,Y)=(G_v,T)$.

Before proceeding we note that for each $v\in S_f$ and $i\geq 2$, the $\cO$-module $C_v^{i}$ is torsion-free. This is clear for $i\geq 3$ and follows from Lemma \ref{tflemma}(i) and the exact sequence
$$0\to H^1(K_v,T_v^-)\to C_v^2\to \rgamma^1(G_v,T_v^-)/\ker(d^1_v)\to0$$
if $i=2$ and $v\in S_p$ because the third term is clearly torsion-free. Similarly, if $i=2$ and $v\in S_f\setminus S_p$, then it follows from Lemma \ref{tflemma}(ii) and the exact sequence
$$0\to H^1_{/{\rm ur}}(K_v,T)\to C_v^2\to \rgamma^1(G_v,T)/\ker(d^1_v)\to0.$$

Now, for each $v\in S_p$, since $H^0(K_v,T_v^-)=0$ {\color{black}by Hypothesis \ref{hyp tech}(ii)}, it is clear that $\rgamma(K_v,T_v^-)$ is isomorphic in $D(\cR)$ to the complex $D^\bullet_v$ with $D_v^i=C_v^{i+1}$ in each degree $i\geq 1$, with differentials induced by the usual coboundary operators $d^{i}_v$, and with $D_v^{i}=0$ for $i\leq 0$. Similarly, for each $v\in S_f\setminus S_p$, since {\color{black}$H^0_{/{\rm ur}}(K_v,T)=0$}, $\rgamma_{/{\rm ur}}(K_v,T)$ is isomorphic in $D(\cR)$ to the complex $D^\bullet_v$ with $D_v^i=C_v^{i+1}$ in each degree $i\geq 1$, with differentials induced by the usual coboundary operators $d^{i}_v$, and with $D_v^{i}=0$ for $i\leq 0$.

Then, since $\rgamma(G_{K,S},T)$ is acyclic in degrees less than one {\color{black} by $H^0(K,T)=0$ (which is implied by $H^0(K,A)=0$ in Hypothesis \ref{hyp tech}(i))}, we may obtain an explicit representative of $C_{\rm Nek}$ as the $(-1)$-shift of the explicit mapping cone of the diagonal localization morphism
$$\tau_{\geq 1}\bigl(\rgamma(G_{K,S},T)\bigr)\to{\bigoplus}_{v\in S_f}D_v^\bullet.$$
(Here $\tau_{\geq 1}$ denotes the truncation in degree greater than or equal to 1 preserving cohomology.)

To be explicit about this construction of a representative of $C_{\rm Nek}$, we set
\[
C_{\rm Nek}^i :=
\begin{cases}
  0,    &\text{ if $i\leq 0$}\\
  \rgamma^1\bigl(G_{K,S},T\bigr)/{\rm im}(d^0),    &\text{ if $i=1$}\\
\bigl(\bigoplus_{v\in S_f}C^i_v\bigr)\oplus \rgamma^i\bigl(G_{K,S},T\bigr),    &\text{ if $i\geq 2$.}
\end{cases}
\]
We let the differential $d_{\rm Nek}^1$ in degree $1$ be given by
\[d_{\rm Nek}^1(\overline{c}):=\bigl(-(\overline{c_{v}})_{v\in S_f},d^1(c)\bigr)\]
for $c\in\rgamma^1(G_{K,S},T)$, and the differential $d_{\rm Nek}^i$ in degree $i\geq 2$ be given by
\[d_{\rm Nek}^i(\bigl((a_{v})_{v\in S_f},b\bigr)):=\bigl(-(d^{i-1}_{v}(a_{v})+b_{v})_{v\in S_f},d^i(b)\bigr)\]
for $a_v\in C_v^{i}$ and $b\in\rgamma^i(G_{K,S},T)$. In any degree $i$ and for any $v\in S_p$, we have also written $x_{v}$ for the image in $\rgamma^{i}(G_v,T_v^-)$ of the localization at $v$
of an inhomogeneous $i$-cochain $x$ in $\rgamma^i\bigl(G_{K,S},T\bigr)$; while, for $v\in S_f\setminus S_p$, we have written $x_{v}\in \rgamma^i(G_v,T)$ for the localization at $v$ of $x$.

Before proceeding we note that each $\cO$-module {\color{black}$C_{\rm Nek}^{i}$} is torsion-free. This is clear for $i\neq 1$ and follows from Lemma \ref{tflemma}(iii) and the exact sequence
$$0\to H^1(G_{K,S},T)\to C_{\rm Nek}^1\to \rgamma^1(G_{K,S},T)/\ker(d^1)\to0$$
if $i=1$ because the third term is clearly torsion-free.
In particular, after setting $\cO_n:=\cO/\pi^n\cO$, the complex $C_{\rm Nek}/\pi^n$ may be represented by the sequence of modules $\cO_n\otimes_{\cO}C_{\rm Nek}^i$ together with the differentials $\cO_n\otimes_{\cO}d_{\rm Nek}^i$.

\subsubsection{Definition of the morphism}\label{step3}

In order to define the claimed morphism $\varphi$, we first define $$\varphi^1:H^1(G_{K,S},A_n)\to \cO_n\otimes_{\cO}C_{\rm Nek}^1$$ as follows.
Given $x\in H^1(G_{K,S},A_n)$ we fix a 1-cocycle $y\in \rgamma^1(G_{K,S},A_n)$ representing $x$ and use the map $l_{K,S}$ defined in Lemma \ref{lifts} (with $(G,X,X_n)=(G_{K,S},T,A_n)$) below to define $\varphi^1(x)$ as the class in $\cO_n\otimes_{\cO}C_{\rm Nek}^1$ of $l_{K,S}(y)\in \cO_n\otimes_{\cO}\rgamma^1(G_{K,S},T)$.

Before proceeding we note that Lemma \ref{lifts} ensures both that $\varphi^1$ is well-defined and also that
\begin{equation}\label{itisacocycle}(\cO_n\otimes_{\cO}d^1)(l_{K,S}(y))=0.
\end{equation}
Moreover, we claim that $\varphi^1$ is injective. To prove this we write $d^0_n:A_n\to \rgamma^1(G_{K,S},A_n)$ for the map given by $(d^0_n(a))(\sigma)=\sigma\cdot a-a$ for any $a\in A_n$ and $\sigma\in G_{K,S}$ and recall that $H^1(G_{K,S},A_n)$ is a subgroup of $\coker(d^0_n)$. We assume that $l_{K,S}(y)$ belongs to $$\cO_n\otimes_{\cO}\im(d^0)=\im(\cO_n\otimes_{\cO}d^0)$$ and proceed to prove that $y$ then belongs to $\im(d^0_n)$, so that $x=0$.

For this we note that every element of ${\color{black}A_n=\cO_n\otimes_{\cO}T}$ is of the form $1\otimes t$ with $t\in T$ and assume that
$$l_{K,S}(y)=(\cO_n\otimes_{\cO}d^0)(1\otimes t)=1\otimes d^0(t)=\xi(d^0(t))=l_{K,S}(\mu(d^0(t)),$$
with the last equality by the commutativity of the diagram in Lemma \ref{lifts}. Since $l_{K,S}$ is injective, it follows that $y=\mu(d^0(t))$ and, writing $\varpi:T\to A_n$ for the canonical projection, we compute for any $\sigma\in G_{K,S}$ that
$$y(\sigma)=(\mu(d^0(t)))(\sigma)={\color{black}\varpi ( d^0(t)(\sigma))}=\varpi(\sigma\cdot t-t)=\sigma\cdot\varpi(t)-\varpi(t)=(d^0_n(\varpi(t)))(\sigma).$$
We conclude that $y=d^0_n(\varpi(t))\in\im(d^0_n)$, as required.

We next define $$\varphi^2:{\bigoplus}_{v\in S_f}H^1_{/\cF}\bigl(K_v,A_n)\to \cO_n\otimes_{\cO}C_{\rm Nek}^2$$ as follows.
Given $(x_{v})_{v\in S_f}$ in
\[{\bigoplus}_{v\in S_f}H^1_{/\cF}\bigl(K_v,A_n)={\bigoplus}_{v\in S_f}\frac{H^1\bigl(K_v,A_n\bigr)}
{H^1_{\cF}(K_v,A_n)}\]
we fix a family of 1-cocycles
$$(y_{v})_{v\in S_f}\in {\color{black} \bigoplus_{v\in S_p} \rgamma^1(G_v,A_{n,v}^-) \oplus \bigoplus_{v\in S_f\setminus S_p}\rgamma^1(G_v,A_n)}$$
representing $(x_{v})_{v\in S_f}$. 
{\color{black}(Here for $v\in S_p$ we regard $x_v \in H^1(K_v,A_{n,v}^-)$ via the natural injection $H^1(K_v,A_n)/ H^1_\cF(K_v,A_n) \hookrightarrow H^1(K_v,A_{n,v}^-)$.)} 
We let $l_v$ denote the map defined by Lemma \ref{lifts} for $(G,X,X_n)=(G_v,T_v^-,A_{n,v}^-)$ if $v\in S_p$ and for $(G,X,X_n)=(G_v,T,A_{n})$ if $v\in S_f\setminus S_p$. 
We then define $\varphi^2((x_{v})_{v\in S_f})$ as the class in $\cO_n\otimes_{\cO}C_{\rm Nek}^2$ of $(-(l_{v}(y_{v}))_{v\in S_f},0)$. Again one can use Lemma \ref{lifts} to verify that $\varphi^2$ is well-defined.
%
%

Our explicit description of the differential $d_{\rm Nek}^1$ now implies that, given $x\in H^1(G_{K,S},A_n)$, one has
\begin{align*} \bigl((\cO_n\otimes_{\cO}d_{\rm Nek}^1)\circ\varphi^1\bigr)(x) &=\bigl(-(\overline{l_{K,S}(y)_{v}})_{v},(\cO_n\otimes_{\cO}d^1)(l_{K,S}(y))\bigr)\\
&=\bigl(-(\overline{l_{K,S}(y)_{v}})_{v},0\bigr)\\
&=\bigl(-(\overline{l_{v}(y_{v})})_{v},0\bigr)\\
&=\varphi^2((\overline{y_{v}})_{v})\\
&=\varphi^2(\lambda(\overline{y}))\\
&=(\varphi^2\circ\lambda)(x).\end{align*}
Here again $y$ is a 1-cocycle representing $x$, the second equality follows from (\ref{itisacocycle}) and the third equality holds because the maps $l_{K,S}$ and $l_{v}$ commute with localization at each $v$. Recall that $\lambda$ denotes the localization map $H^1(G_{K,S},A_n)\to  \bigoplus_{v\in S_f}H^1_{/\cF}(K_v,A_n)$.

Putting then $\varphi^i=0$ for $i\neq 1,2$ therefore defines a morphism of complexes of $\cR_n$-modules $$\varphi:C_{{\rm PT}}\to C_{\rm Nek}/\pi^n.$$ Indeed, using Lemma \ref{lifts} one verifies that
\[ \bigl((\cO_n\otimes_{\cO}d_{\rm Nek}^2)\circ\varphi^2\bigr)((x_{v})_{v})=0\] for any $(x_{v})_{v}$ in ${\bigoplus}_{v\in S_f}H^1_{/\cF}\bigl(K_v,A_n)$.

\subsubsection{Commutativity of the diagrams}\label{step4}

We next prove that the given diagrams in Theorem \ref{qithm} commute. In fact, the commutativity of the first diagram is immediate from the explicit definition of the maps involved and from the commutativity of the triangle in Lemma \ref{lifts}.

Turning now to the second diagram, for $X\in\{T\}\cup\{A_m\}_{m\in\NN}$ we set
$$\rgamma_c(G_{K,S},X):= {\rm Cone}\left(\rgamma(G_{K,S},X)\to \bigoplus_{v\in S_f}\rgamma(K_v,X)\right)[-1].$$
In particular, we derive a canonical exact triangle in $D(\cR)$ of the form
\begin{equation}\label{compacttriangle}\rgamma_c(G_{K,S},T)\stackrel{\alpha}{\longrightarrow} C_{\rm Nek}\longrightarrow{\bigoplus}_{v\in S_p}\rgamma(K_v,T_v^+)\oplus{\bigoplus}_{v\in S_f}\rgamma_{\rm ur}(K_v,T).\end{equation}

For each $m\geq 1$, we will use the canonical isomorphism
$$\sigma_m:H^2_c(G_{K,S},A_m)\cong H^1(G_{K,S},A_m^*(1))^*$$
{\color{black}which is induced by the morphism 
$$\widetilde \sigma_m: \rgamma_c(G_{K,S},A_m)\to \rhom_{\cO_m}(\rgamma(G_{K,S},A_m^\ast(1)) , \cO_m)[-3]$$
defined in \cite[\S 5.4.2]{nekovar}.}
The limit $\varprojlim_m\sigma_m$ then defines an isomorphism
$$\sigma:H^2_c(G_{K,S},T)\cong\varprojlim_m H^2_c(G_{K,S},A_m)\cong\varprojlim_m H^1(G_{K,S},A_m^*(1))^*\cong H^1(G_{K,S},T^\vee(1))^\vee.$$
By the second diagram in \cite[Proposition 6.3.3]{nekovar}, the global duality isomorphism in Lemma \ref{lem sel}(ii) makes the following square commute
\begin{equation}\label{gdsquare}\begin{CD}
H^2_c(G_{K,S},T) @> \sigma >> H^1(G_{K,S},T^\vee(1))^\vee\\
@V H^2(\alpha) VV @VV \iota^\vee V\\
H^2(C_{\rm Nek}) @> \sim >> H^1_{\cF^\ast}(K,T^\vee(1))^\vee.
\end{CD}\end{equation}
Here the vertical arrows are induced by the exact triangle (\ref{compacttriangle}) and by the dual of the inclusion $\iota:H^1_{\cF^\ast}(K,T^\vee(1))\hookrightarrow H^1(G_{K,S},T^\vee(1))$, respectively.

We also recall the explicit definition of the map $\nu$ that occurs in the exact sequence (\ref{PT}). The local Tate duality pairings define an isomorphism
$$u:{\bigoplus}_{v\in S_f}H^1(K_v,A_n)\cong{\bigoplus}_{v\in S_f}H^1(K_v,A_n^*(1))^*$$ that, by \cite[Theorem 2.3.4(ii)]{MRkoly}, induces an isomorphism
$$\overline{u}:{\bigoplus}_{v\in S_f}H^1_{/\cF}(K_v,A_n)\cong{\bigoplus}_{v\in S_f}H^1_{\cF^*}(K_v,A_n^*(1))^*.$$
Denoting by $\Lambda$ the localization map $$H^1_{\cF^*}(K,A_n^*(1))\to{\bigoplus}_{v\in S_f}H^1_{\cF^*}(K_v,A_n^*(1)),$$ one then has $\nu=\Lambda^*\circ\overline{u}$. We abuse notation to regard $\nu$ also as the isomorphism 
\begin{equation}\label{abuse nu}
H^2(C_{\rm PT})=\coker(\lambda)\cong H^1_{\cF^*}(K,A_n^*(1))^*
\end{equation}
that is induced by $\nu$ (under the validity of Hypothesis \ref{hyp surj}). 


The second diagram in \cite[Proposition 6.3.3]{nekovar} (with the ``relaxed" local condition, namely, $U_S^+(A_n):= \bigoplus_{v\in S_f}\rgamma(K_v,A_n)$) gives rise to the following commutative diagram with exact rows:
\begin{equation}\label{duality diagram}
\xymatrix{
\rgamma_c(G_{K,S},A_n) \ar[r] \ar[d]_{\widetilde \sigma_n}&  \rgamma(G_{K,S},A_n) \ar[d]\ar[r]&\bigoplus_{v\in S_f}\rgamma(K_v,A_n)\ar[d]^{\widetilde u} \\
\rgamma(G_{K,S},A_n^\ast(1))^\ast[-3] \ar[r]& \rgamma_c(G_{K,S},A_n^\ast(1))^\ast[-3] \ar[r]& \bigoplus_{v\in S_f} \rgamma(K_v,A_n^\ast(1))^\ast[-2],
}
\end{equation}
where (by an abuse) we abbreviate $\rhom_{\cO_n}(-,\cO_n)$ to $(-)^\ast$, and $\widetilde \sigma_n$ and $\widetilde u$ are the duality morphisms which induce $\sigma_n$ and $u$ respectively. 
This diagram implies that the following square commutes:
\begin{equation}\label{lgsquare}\begin{CD}
{\bigoplus}_{v\in S_f}H^1(K_v,A_n) @> u >> {\bigoplus}_{v\in S_f}H^1(K_v,A_n^*(1))^*\\
@V \delta VV @VV \Lambda^* V\\
H^2_c(G_{K,S},A_n) @> \sigma_n >> H^1(G_{K,S},A_n^*(1))^*.
\end{CD}\end{equation}
Here $\delta$ is the canonical connecting homomorphism associated with the first row in (\ref{duality diagram}), 
and we have abused notation to also write $\Lambda$ for the localization map $$H^1(G_{K,S},A_n^*(1))\to{\bigoplus}_{v\in S_f}H^1(K_v,A_n^*(1)).$$

Note that the map
$$H^2(\alpha): H^2_c(G_{K,S},T)\to H^2(C_{\rm Nek})$$
is surjective, since we have $H^2(K_v,T_v^+) \simeq H^0(K_v,(T_v^+)^\vee(1))^\vee =0$ for all $v\in S_p$ by Hypothesis \ref{hyp tech}(iii).  We then abuse notation to write $H^2(\alpha)^{-1}$ for the inverse of the isomorphism
$$\frac{H^2_c(G_{K,S},T)}{\im\left({\bigoplus}_{v\in S_f}H^1_{\cF}(K_v,T)\right)}\cong H^2(C_{\rm Nek})$$
that is induced by $H^2(\alpha)$.

Then an explicit computation using our definition of $\varphi^2$ (and the commutativity of the triangle in Lemma \ref{lifts}) shows that the composition
\begin{multline*}\bigoplus_{v\in S_f}H^1(K_v,A_n)\to \bigoplus_{v\in S_f}H^1_{/\cF}(K_v,A_n)\to H^2(C_{{\rm PT}}) \xrightarrow{H^2(\varphi)}H^2(C_{\rm Nek}/\pi^n)\xrightarrow{q_{2}^{-1}}H^2(C_{\rm Nek})/\pi^n\\ \xrightarrow{H^2(\alpha)^{-1}/\pi^n}\frac{H^2_c(G_{K,S},T)/\pi^n}{\im\left({\bigoplus}_{v\in S_f}H^1_{\cF}(K_v,T)/\pi^n\right)}\to\frac{H^2_c(G_{K,S},A_n)}{\im\left({\bigoplus}_{v\in S_f}H^1_{\cF}(K_v,A_n)\right)}\end{multline*}
coincides with $-1$ times the canonical map that is induced by the connecting homomorphism $\delta$ (that occurs in (\ref{lgsquare})).

From the description $\nu=\Lambda^*\circ\overline{u}$ and the commutativity of (\ref{lgsquare}) it thus follows that the composition
\begin{multline*}
{\color{black}\bigoplus_{v\in S_f}H^1_{/\cF}(K_v,A_n)\to} H^2(C_{{\rm PT}}) \xrightarrow{H^2(\varphi)}H^2(C_{\rm Nek}/\pi^n)\xrightarrow{q_{2}^{-1}}H^2(C_{\rm Nek})/\pi^n \\
\xrightarrow{H^2(\alpha)^{-1}/\pi^n}\frac{H^2_c(G_{K,S},T)/\pi^n}{\im\left({\bigoplus}_{v\in S_f}H^1_{\cF}(K_v,T)/\pi^n\right)} \to\frac{H^2_c(G_{K,S},A_n)}{\im\left({\bigoplus}_{v\in S_f}H^1_{\cF}(K_v,A_n)\right)}\\
\xrightarrow{\sigma_n}\frac{H^1(G_{K,S},A_n^*(1))^*}{\im\left({\bigoplus}_{v\in S_f}H^1_{/\cF^*}(K_v,A_n^*(1))^*\right)}\xrightarrow{\iota_n^*} H^1_{\cF^*}(K,A_n^*(1))^*\end{multline*}
is equal to $-\nu$. Here we have written $\iota_n^*$ for the map that is induced by the dual of the inclusion $H^1_{\cF^*}(K,A_n^*(1))\hookrightarrow H^1(G_{K,S},A_n^*(1))$.

{\color{black}Since $\kappa_2$ is by definition induced by the bottom arrow in the commutative diagram (\ref{gdsquare}), }
one finally finds that 
\begin{equation}\label{H2}
\kappa_2 \circ q_2^{-1}\circ H^2(\varphi)=-\nu,
\end{equation}
where we regard $\nu$ as the isomorphism (\ref{abuse nu}). This proves the commutativity of the second diagram in Theorem \ref{qithm}.

\subsubsection{Completion of the proof}\label{step5}

We finally prove that $\varphi$ is a quasi-isomorphism. We recall that $H^1(C_{\rm PT})=\ker(\lambda)=H^1_{\cF}(K,A_n)$. Since also $C_{\rm Nek}^0=0$, the map $H^1(\varphi)$ is simply the restriction of $\varphi^1$ to $\ker(\lambda)$, and is injective because $\varphi^1$ is (as was proved above: {\color{black}see the discussion after (\ref{itisacocycle})}). Now, since $\# H^1(C_{\rm PT})=\# H^1(C_{\rm Nek}/\pi^n)$
by Lemma \ref{lem NekPT}(iii) (and since these orders are finite), $H^1(\varphi)$ is necessarily bijective.

Since $C_{\rm Nek}/\pi^n$ is acyclic outside degrees one and two, it only remains to show that $H^2(\varphi)$ is bijective, and this is now an immediate consequence of (\ref{H2}). \qed

\section{Stark systems}

In this section, we give applications of Theorem \ref{qithm} to the theory of Stark systems. 
Let $p,K,S, \Phi,\cO,\pi, \cR, T$ be as in \S \ref{sec nek selmer}.

\subsection{The module of Stark systems}\label{sec stark}

Fix a positive integer $n$. We set
$$\cR_n:=\cR/\pi^n\cR \text{ and }A_n:=T/\pi^n T.$$
Fix a Selmer structure $\cF$ on $A_n$. (For the moment, we do not need to assume $\cF$ is of the special form as in Definition \ref{def sel str}.)
Let $\cP=\cP_n$ be a set of primes $\fq$ of $K$ such that $\fq \notin S_\infty\cup S_p\cup S_{\rm ram}(T)$ (where $S_{\rm ram}(T)$ denotes the set of places of $K$ at which $T$ ramifies) and 
$$H^1_{/\cF}(K_\fq, A_n)\simeq \cR_n.$$
Let $\cN=\cN(\cP)$ be the set of square-free products of primes in $\cP$. For $\fn\in \cN$, let $\nu(\fn)$ be the number of prime divisors of $\fn$.

For a commutative ring $R$ and an $R$-module $M$, let
$${\bigcap}_R^r M:= \left( {\bigwedge}_R^r (M^\ast)\right)^\ast$$
be the $r$-th exterior bidual as in \cite{sbA}. 

\begin{lemma}\label{lem transition}
Let 
$$0\to X\to Y \to Z$$
be an exact sequence of finitely generated $\cR_n$-modules. Suppose that $Z$ is free of rank $s$. Then for any $r \geq 0$ the exact sequence induces a canonical homomorphism
$${\bigcap}_{\cR_n}^{r+s} Y\to \left({\bigcap}_{\cR_n}^r X \right) \otimes_{\cR_n} {\det}_{\cR_n}(Z).$$
\end{lemma}

\begin{proof}
Fix a basis $\{z_1,\ldots,z_s\}$ of $Z$ and regard $Z=\cR_n^s$. By \cite[Proposition A.3]{sbA}, we know that the exact sequence
$$0\to X\to Y\to \cR_n^s$$
induces a map
$$\Phi: {\bigcap}_{\cR_n}^{r+s} Y\to {\bigcap}_{\cR_n}^r X.$$
One sees that the map
$${\bigcap}_{\cR_n}^{r+s} Y\to \left({\bigcap}_{\cR_n}^r X \right) \otimes_{\cR_n} {\det}_{\cR_n}(Z); \ y \mapsto \Phi(y)\otimes (z_1\wedge\cdots \wedge z_s)$$
is independent of the choice of the basis $\{z_1,\ldots,z_s\}$. 
\end{proof}

For $\fm,\fn\in \cN$ with $\fm \mid \fn$, by applying Lemma \ref{lem transition} to the exact sequence
$$0\to H^1_{\cF^\fm}(K,A_n)\to H^1_{\cF^\fn}(K,A_n)\to \bigoplus_{\fq \mid \fn/\fm}H^1_{/\cF}(K_\fq,A_n),$$
we obtain a canonical map
$$ {\bigcap}_{\cR_n}^{r+\nu(\fn)} H^1_{\cF^\fn}(K,A_n)\to {\bigcap}_{\cR_n}^{r+\nu(\fm)} H^1_{\cF^\fm}(K,A_n)\otimes_{\cR_n} \bigotimes_{\fq \mid \fn/\fm}H^1_{/\cF}(K_\fq,A_n) $$
for any $r\geq 0$. This map induces
$$v_{\fn,\fm}: {\bigcap}_{\cR_n}^{r+\nu(\fn)} H^1_{\cF^\fn}(K,A_n)\otimes_{\cR_n} \bigotimes_{\fq \mid \fn}H^1_{/\cF}(K_\fq,A_n)^\ast\to {\bigcap}_{\cR_n}^{r+\nu(\fm)} H^1_{\cF^\fm}(K,A_n)\otimes_{\cR_n} \bigotimes_{\fq \mid \fm}H^1_{/\cF}(K_\fq,A_n)^\ast. $$

\begin{definition}[{See \cite[\S 4.1]{bss1}}]\label{def stark}
Let $r\geq 0$. The module of Stark systems of rank $r$ for $(A_n,\cF,\cP) $ is defined by
$${\rm SS}_r(A_n,\cF,\cP) := \varprojlim_{\fn \in \cN}\left( {\bigcap}_{\cR_n}^{r+\nu(\fn)} H^1_{\cF^\fn}(K,A_n) \otimes_{\cR_n} \bigotimes_{\fq \mid \fn} H^1_{/\cF}(K_\fq,A_n)^\ast\right),$$
where the inverse limit is taken with respect to $v_{\fn,\fm}$. 
\end{definition}

Let $C_{\rm PT}$ be the Poitou-Tate complex for $(A_n,\cF)$ in Definition \ref{def PT} (see also Remark \ref{rem abb}). Recall from Proposition \ref{prop core}(ii) that $C_{\rm PT}$ is a perfect complex of $\cR_n$-modules if we assume the existence of a core vertex for $\cF$ (see Definition \ref{def core}).

\begin{proposition}\label{prop stark}
Assume $\cN$ contains a core vertex for $\cF$. Then there is a canonical isomorphism
$${\det}_{\cR_n}^{-1}(C_{\rm PT}) \simeq {\rm SS}_{\chi(\cF)}(A_n,\cF,\cP). $$
In particular, the $\cR_n$-module ${\rm SS}_{\chi(\cF)}(A_n,\cF,\cP)$ is free of rank one. 
\end{proposition}

\begin{proof}
Let $\fn \in \cN$ be a core vertex. Take any $\fm \in \cN$ that is coprime to $\fn$. We first claim that $\fd:=\fn \fm$ is also a core vertex. To prove this, we check that $H^1_{\cF^\fd}(K,A_n)$ is a free $\cR_n$-module and $H^1_{(\cF^\fd)^\ast}(K,A_n^\ast(1))=0$. The latter condition holds since $H^1_{(\cF^\fn)^\ast}(K,A_n^\ast(1))=0$ and $\fn \mid \fd$. To show that $H^1_{\cF^\fd}(K,A_n)$ is $\cR_n$-free, consider the exact sequence
$$0\to H^1_{\cF^\fn}(K,A_n)\to H^1_{\cF^\fd}(K,A_n) \to \bigoplus_{\fq \mid \fm}H^1_{/\cF}(K_\fq,A_n).$$
The last map is surjective since $H^1_{(\cF^\fn)^\ast}(K,A_n^\ast(1))=0$. Since $H^1_{/\cF}(K_\fq,A_n)\simeq \cR_n$ for each $\fq \in \cP$ and $H^1_{\cF^\fn}(K,A_n)$ is $\cR_n$-free, the claim follows. 

We define a homomorphism
$$\psi: {\det}_{\cR_n}^{-1}(C_{\rm PT}) \to {\rm SS}_{\chi(\cF)}(A_n,\cF,\cP)$$
as follows. Take any $\fz \in {\det}_{\cR_n}^{-1}(C_{\rm PT}) $ and $\fm \in \cN$. It is sufficient to define $\psi(\fz)_\fm$. The above claim implies that there is a core vertex $\fn \in \cN$ such that $\fm \mid \fn$. By (\ref{core formula}) we have $\chi(\cF)+\nu(\fn)={\rm rank}_{\cR_n}(H^1_{\cF^\fn}(K,A_n))$ so, by Proposition \ref{prop core}(ii), we have a canonical identification
$${\det}_{\cR_n}^{-1}(C_{{\rm PT}})= {\bigcap}_{\cR_n}^{\chi(\cF)+\nu(\fn)}H^1_{\cF^\fn}(K,A_n) \otimes_{\cR_n}\bigotimes_{\fq\mid \fn}H^1_{/\cF}(K_\fq,A_n)^\ast.$$
We then define 
$$\psi(\fz)_{\fm}:= v_{\fn,\fm}(\fz).$$
One checks easily that the map $\psi$ is well-defined. Since the map
$${\rm SS}_{\chi(\cF)}(A_n,\cF,\cP) \to  {\bigcap}_{\cR_n}^{\chi(\cF)+\nu(\fn)}H^1_{\cF^\fn}(K,A_n) \otimes_{\cR_n}\bigotimes_{\fq\mid \fn}H^1_{/\cF}(K_\fq,A_n)^\ast; \ \epsilon \mapsto \epsilon_\fn$$
is bijective by \cite[Theorem 4.6(i)]{bss1}, we see that $\psi$ is an isomorphism. 
\end{proof}

Combining Proposition \ref{prop stark} with Theorem \ref{qithm}, we obtain the following. 

\begin{theorem}\label{thm stark}
We fix the data $\sF=(\sF_v)_{v\in S_p}$ as in (\ref{greenberg local}) and let $\widetilde \rgamma_\sF(K,T)$ be the associated Selmer complex (see Definition \ref{def sel}). Let $\cF$ be the Selmer structure attached to $\sF$ (see Definition \ref{def sel str}).
Assume Hypothesis \ref{hyp tech} and that $\cN$ contains a core vertex for $\cF$ on $A_n$ with $\chi(\cF)=\chi(\sF)$ (so Hypothesis \ref{hyp core} is also satisfied). Then there is a canonical isomorphism
$$\varpi_n: {\det}_{\cR}^{-1}(\widetilde \rgamma_\sF(K,T))/\pi^n = {\det}_{\cR_n}^{-1}(\cO/\pi^n\cO \lotimes_\cO \widetilde \rgamma_\sF(K,T)) \xrightarrow{\sim} {\rm SS}_{\chi(\sF)}(A_n,\cF,\cP). $$
\end{theorem}

Applying the general theory of Stark systems in \cite[Theorem 4.6]{bss1}, we have the following. 

\begin{corollary}\label{cor stark}
Assume the same hypotheses as in Theorem \ref{thm stark}. Suppose that an $\cR$-basis
$$\fz \in {\det}_\cR^{-1}(\widetilde \rgamma_\sF(K,T))$$
is given. 
We set
$$\epsilon:= \varpi_n(\fz) \in {\rm SS}_{\chi(\sF)}(A_n,\cF,\cP)$$ 
where $\varpi_n$ is the isomorphism in Theorem \ref{thm stark}. For each non-negative integer $i$, we follow \cite[Definition 4.1]{bss1} to define an ideal
$$I_i(\epsilon):=\sum_{\fn \in \cN, \ \nu(\fn)=i} \im(\epsilon_\fn).$$
of $\cR_n$. Then we have
$${\rm Fitt}_{\cR_n}^i\left(H^1_{\cF^\ast}(K,A_n^\ast(1))^\ast \right) = I_i(\epsilon) .$$
\end{corollary}

\subsection{Refined Iwasawa theory for Heegner points}\label{sec heeg}
In this subsection, we apply Corollary \ref{cor stark} to the setting of Heegner points and give an analogue of ``refined Iwasawa theory" of Kurihara \cite{kurihara}.

In this subsection, we assume $p\geq 5$. 
Let $E$ be an elliptic curve defined over $\QQ$ which has good ordinary reduction at $p$. Let $N$ be the conductor of $E$. Let $K$ be an imaginary quadratic field satisfying the Heegner hypothesis: every prime divisor of $N$ splits in $K$. We assume $p$ is unramified in $K$. Let $K_\infty/K$ be the anticyclotomic $\ZZ_p$-extension. Let $K_n$ be the $n$-th layer of $K_\infty/K$. We set $\Gamma:=\Gal(K_\infty/K)$ and $\Lambda:=\ZZ_p[[\Gamma]]$. We set $T:=T_p(E)$ and consider the deformation $\TT:=\varprojlim_n {\rm Ind}_{K_n/K}(T)=T\otimes_{\ZZ_p}\Lambda$. Since $E$ has good ordinary reduction at $p$, we have a canonical exact sequence
$$\sF_v: 0\to \TT_v^+ \to \TT \to \TT_v^-\to 0$$
for each prime $v$ of $K$ lying above $p$ (see Example \ref{ex greenberg}). Let $\widetilde \rgamma_\sF(K,\TT)$ be the associated Selmer complex as in Definition \ref{def sel}.

We assume the following. 

\begin{hypothesis}\label{hyp heeg}\
\begin{itemize}
\item[(i)] The representation $\rho: G_K\to {\rm Aut}(E[p])$ is surjective.
\item[(ii)] $E(K_v)[p]=0$ for every $v\mid N$. 
\item[(iii)] $N$ is square-free.
\item[(iv)] $p$ is non-anomalous, i.e., $H^0(K_v,\widetilde E_v[p])=0$ for any $v\mid p$. 
\end{itemize}
\end{hypothesis}

Note that Hypothesis \ref{hyp heeg} implies Hypothesis \ref{hyp tech} for ${\rm Ind}_{K_m/K}(T)$ for any $m\geq 0$.

We now fix $n\geq 1$ and $m\geq 0$. We set $\cR_{n,m}:=\ZZ/p^n\ZZ [\Gamma_m] $, 
where $\Gamma_m:=\Gal(K_m/K)$. 
We set $A_n:=T/p^nT\simeq E[p^n]$ and $A_{n,m}:={\rm Ind}_{K_m/K}(A_n)$. Let $\cP_{n,m}$ be the set of primes $\fq$ of $K$ such that $\fq \nmid pN$ and 
$$H^1_{/{\rm ur}}(K_\fq, A_{n,m}) \simeq \cR_{n,m}.$$
Let $\cN_{n,m}$ be the set of square-free products of primes in $\cP_{n,m}$. 

Let $\cF$ be the Selmer structure on $A_{n,m}$ attached to $\sF$ (see Definition \ref{def sel str}). 
Let ${\rm Sel}_{p^n}(E/K_m)$ be the $p^n$-Selmer group for $E/K_m$.

\begin{proposition}\label{prop canonical selmer}
Assume Hypothesis \ref{hyp heeg}(iv). 
Then we have canonical identifications
$${\rm Sel}_{p^n}(E/K_m)= H^1_\cF(K,A_{n,m})=H^1_{\cF^\ast}(K,A_{n,m}^\ast(1)).$$
\end{proposition}

\begin{proof}
Note that we have a canonical identification $A_{n,m} = A_{n,m}^\ast(1)$ via the Weil pairing. 
It is sufficient to check that 
$$\im \left(\bigoplus_{w\mid v}E(K_{m,w})/p^n \xrightarrow{\kappa} \bigoplus_{w\mid v}H^1(K_{m,w},E[p^n])\simeq H^1(K_v,A_{n,m}) \right) = H^1_\cF(K_v,A_{n,m})$$
for any $v\mid pN$, where $\kappa$ denotes the Kummer map. If $v\mid p$, this follows from \cite[Proposition 2.5]{greenberg}, since $p$ is non-anomalous. If $v\nmid p$, then we have $H^1_\cF(K_v,A_{n,m}) = H^1_f(K_v,A_{n,m})$ by definition, and the claim is well-known. 
\end{proof}

\begin{lemma}\label{lem hyp}
Assume Hypothesis \ref{hyp heeg}. Then there exists a core vertex for $\cF$ on $A_{n,m}$ (see Definition \ref{def core}) that belongs to $\cN_{n,m}$. Moreover, we have $\chi(\cF)=\chi(\sF)=0$. (In particular, Hypothesis \ref{hyp core} is satisfied.)
\end{lemma}

\begin{proof}
We first prove the existence of a core vertex. 
Since $p\geq 5$, Hypothesis \ref{hyp heeg}(i) implies that the image of the representation
$$ G_K \to {\rm Aut}(T)\simeq {\rm GL}_2(\ZZ_p)$$
contains ${\rm SL}_2(\ZZ_p)$ (see \cite[Lemma 3 on page IV-23]{serreladic}). Hence, by \cite[Lemma 6.17(ii)]{bss2}, the hypotheses $({\rm H}_1)$-$({\rm H}_3)$ in \cite[\S 3.1]{bss2} are satisfied. By Lemma \ref{tflemma} and \cite[Theorem 3.14]{bss2} (see also \cite[Proposition 3.18]{bss2}), it is sufficient to check that, for any $v\mid pN$, the map
$$H^1_{/\cF}(K_v,A_1)=H^1_{/\cF}(K_v,E[p]) \to H^1_{/\cF}(K_v, {\rm Ind}_{K_m/K}(E[p]))=H^1_{/\cF}(K_v,A_{1,m})$$
induced by $\FF_p\to \FF_p[\Gamma_m]; \ 1\mapsto \sum_{\sigma\in \Gamma_m}\sigma$ is injective.

If $v\nmid p$, then we have $E(K_v)[p]=0$ by Hypothesis \ref{hyp heeg}(ii), and \cite[Lemma 3.10]{bss2} implies that $H^1_{\cF}(K_v,E[p])=H^1_{f}(K_v,E[p])=H^1(K_v,E[p])$, so $H^1_{/\cF}(K_v,E[p])=0$ and
the injectivity is trivial.

If $v\mid p$, the validity of Hypothesis \ref{hyp tech}(iii) for $T$ implies that $H^2(K_v,T_v^+)\simeq H^0(K_v,(T_v^+)^\vee(1))^\vee =0$ and hence also that $H^2(K_v,A^+_{1,v})$ vanishes. Similarly, the validity of Hypothesis \ref{hyp tech}(iii) for ${\rm Ind}_{K_m/K}(T)$ implies that $H^2(K_v,A^+_{1,m,v})$ vanishes.
It is thus sufficient to check that the map
$$H^1(K_v,A^-_{1,v})=H^1(K_v,\widetilde E_v[p]) \to H^1(K_v, {\rm Ind}_{K_m/K}(\widetilde E_v[p]))=H^1(K_v,A^-_{1,m,v})$$
is injective, where $\widetilde E_v$ denotes the reduction of $E$ modulo $v$. By Shapiro's lemma and the Inflation-Restriction sequence, we are reduced to proving that
$$H^1(K_{m,w}/K_v, H^0(K_{m,w}, \widetilde E_v[p]))=0,$$
where $w$ is a place of $K_m$ lying above $v$. 
However, we have $H^0(K_v,\widetilde E_v[p])=0$ by Hypothesis \ref{hyp heeg}(iv), which implies $H^0(K_{m,w},\widetilde E_v[p])=0$ since $K_{m,w}/K_v$ is a $p$-extension. Hence we have proved the first claim. 

To prove the last claim, we first note that $\chi(\sF)=0$ in the present setting (see Proposition \ref{prop perf}(ii)). So it is sufficient to show $\chi(\cF)=0$. If $\fn$ is a core vertex, then by the Poitou-Tate exact sequence and Proposition \ref{prop canonical selmer} we have an exact sequence
$$0\to {\rm Sel}_{p^n}(E/K_m) \to H^1_{\cF^\fn}(K,A_{n,m}) \to \bigoplus_{\fq \mid \fn} H^1_{/\cF}(K_\fq, A_{n,m}) \to {\rm Sel}_{p^n}(E/K_m)^\ast \to 0.$$
Since the modules appearing in this sequence are all finite, we have 
$$\# H^1_{\cF^\fn}(K,A_{n,m})= \prod_{\fq \mid \fn}\# H^1_{/\cF}(K_\fq,A_{n,m}).$$
Hence we have $\chi(\cF)=0$ by the definition (\ref{core formula}). This completes the proof. 
\end{proof}

Under Hypothesis \ref{hyp heeg}, Perrin-Riou's ``Heegner point main conjecture" is proved in \cite{BCK}. Combining this result with \cite[Proposition 5.12]{ks}, we obtain the following. 

\begin{theorem}\label{HIMC}
Assume Hypothesis \ref{hyp heeg}. Let ${\rm Sel}(\TT)$ be the $\Lambda$-adic Selmer group and $y_\infty \in {\rm Sel}(\TT)$ the system of Heegner points as in \cite[\S 5.2.1]{ks}. 
Let
$$\pi: Q(\Lambda)\otimes_\Lambda {\det}_\Lambda^{-1}(\widetilde \rgamma_\sF(K,\TT)) \xrightarrow{\sim} Q(\Lambda )\otimes_\Lambda {\rm Sel}(\TT)^{\otimes 2}$$
be the canonical isomorphism in \cite[(5.2.1)]{ks}, where $Q(\Lambda)$ denotes the quotient field of $\Lambda$. Then there is a (unique) $\Lambda$-basis
$$\fz_\infty^{\rm Hg} \in {\det}_\Lambda^{-1}(\widetilde \rgamma_\sF(K,\TT)) $$
such that $\pi(\fz_\infty^{\rm Hg})=y_\infty\otimes y_\infty$. 
\end{theorem}

Using the element $\fz_\infty^{\rm Hg}$ in Theorem \ref{HIMC}, we shall construct a ``Heegner point Stark system" in the following way. 

In the following, we assume Hypothesis \ref{hyp heeg} (which implies Hypothesis \ref{hyp tech}). 
By Lemma \ref{lem hyp}, we see that Hypothesis \ref{hyp core} is satisfied. Hence we can apply Theorem \ref{thm stark} to obtain a canonical isomorphism
$$\varpi_{n,m}: {\det}_{\ZZ_p[\Gamma_m]}^{-1}(\widetilde \rgamma_\sF(K,{\rm Ind}_{K_m/K}(T)))/p^n \xrightarrow{\sim} {\rm SS}_0(A_{n,m}):= {\rm SS}_0(A_{n,m},\cF,\cP_{n,m}).$$

Let
$$\fz_m^{\rm Hg} \in {\det}_{\ZZ_p[\Gamma_m]}^{-1}(\widetilde \rgamma_\sF(K,{\rm Ind}_{K_m/K}(T)))$$
be the image of $\fz_\infty^{\rm Hg}$ under the natural surjection
$${\det}_\Lambda^{-1}(\widetilde \rgamma_\sF(K,\TT))\twoheadrightarrow {\det}_\Lambda^{-1}(\widetilde \rgamma_\sF(K,\TT)) \otimes_{\Lambda}\ZZ_p[\Gamma_m]\simeq {\det}_{\ZZ_p[\Gamma_m]}^{-1}(\widetilde \rgamma_\sF(K,{\rm Ind}_{K_m/K}(T))),$$
where the last isomorphism is due to \cite[Proposition 8.10.10]{nekovar}. 

\begin{definition}\label{def heeg stark}
We define the {\it Heegner point Stark system} (for $A_{n,m}$) by
$$\epsilon_{n,m}^{\rm Hg} := \varpi_{n,m}( \fz_m^{\rm Hg}) \in {\rm SS}_0(A_{n,m}).$$
\end{definition}

\begin{remark}
Since $\fz_\infty^{\rm Hg}$ is a $\Lambda$-basis, we see that $\epsilon_{n,m}^{\rm Hg}$ is an $\cR_{n,m}$-basis. 
\end{remark}

By Proposition \ref{prop canonical selmer} and Corollary \ref{cor stark}, we obtain the following. 

\begin{theorem}\label{thm heeg}
Assume Hypothesis \ref{hyp heeg}. 
Then for any $n\geq 1$, $m\geq 0$ and $i\geq 0$ we have
$${\rm Fitt}_{\cR_{n,m}}^i({\rm Sel}_{p^n}(E/K_m)^\ast) = I_i(\epsilon_{n,m}^{\rm Hg}).$$
\end{theorem}

\begin{remark}\label{rem koly}
Theorem \ref{thm heeg} can be seen as an equivariant generalization of the classical structure theorem of Kolyvagin \cite{kolyvagin1}, \cite{kolyvagin2} (see also \cite[Theorem 4.7]{zhang}). 
\end{remark}

%
%

\begin{remark}\label{rem kurihara}
Our method can also be applied to the cyclotomic setting as in \cite{kurihara}. To explain this, let $\QQ_\infty/\QQ$ be the cyclotomic $\ZZ_p$-extension and $\Lambda$ the associated Iwasawa algebra. Let $\TT:=T\otimes_{\ZZ_p}\Lambda$ be the cyclotomic deformation of $T:=T_p(E)$. Let $\cL^{\rm MSD} \in Q(\Lambda)$ be the $p$-adic $L$-function of Mazur-Swinnerton-Dyer. Then Mazur's cyclotomic Iwasawa main conjecture is equivalent to the following: $Q(\Lambda)\lotimes_\Lambda\widetilde \rgamma_\sF(\QQ,\TT)$ is acyclic and there is a (unique) $\Lambda$-basis
$$\fz_\infty^{\rm MSD} \in {\det}_\Lambda^{-1}(\widetilde \rgamma_\sF(\QQ,\TT))$$
such that the canonical isomorphism
$$Q(\Lambda)\otimes_\Lambda {\det}_\Lambda^{-1}(\widetilde \rgamma_\sF(\QQ,\TT))\xrightarrow{\sim}Q(\Lambda)$$
sends $\fz_\infty^{\rm MSD}$ to $\cL^{\rm MSD}$. Using this element, we can construct a Stark system
$$\epsilon_{n,m}^{\rm MSD} \in {\rm SS}_0(A_{n,m})$$
exactly in the same way: we have a canonical isomorphism
$$\varpi_{n,m}^{\rm cyc}: {\det}_{\ZZ_p[\Gamma_m]}^{-1}(\widetilde \rgamma_\sF(\QQ,{\rm Ind}_{\QQ_m/\QQ}(T)))/p^n \xrightarrow{\sim} {\rm SS}_0(A_{n,m})$$
 by Theorem \ref{thm stark} and we define $\epsilon_{n,m}^{\rm MSD}:=\varpi_{n,m}^{\rm cyc}(\fz_m^{\rm MSD})$. We have the analogue of Theorem \ref{thm heeg}, namely, 
$${\rm Fitt}_{\cR_{n,m}}^i ({\rm Sel}_{p^n}(E/\QQ_m)^\ast) = I_i(\epsilon_{n,m}^{\rm MSD}).$$
In particular, letting $m=0$ and taking the limit with respect to $n$, we obtain
$${\rm Fitt}_{\ZZ_p}^i({\rm Sel}_{p^\infty}(E/\QQ)^\vee) = \Theta_i^{\rm MSD}:= \varprojlim_n I_i(\epsilon_{n,0}^{\rm MSD}). $$
The latter equality is of the same form as \cite[Theorem B]{kurihara}. 
(However, it is non-trivial that $\Theta_i^{\rm MSD}$ coincides with $\Theta_i(\QQ)$ in \cite[Theorem B]{kurihara}.) 
\end{remark}

\section{Derived $p$-adic heights}

In this section, we compare derived $p$-adic heights of Bertolini-Darmon \cite{BD der} with those of Nekov\'a\v{r} \cite{nekovar}. 

\subsection{Statement}\label{sec statement}

Let us state the main result of this section. 

Let $p$ be an odd prime number. Let $K$ be a number field and $K_\infty/K$ a $\ZZ_p$-extension. We set $\Gamma:=\Gal(K_\infty/K)$ and $\Lambda:= \ZZ_p[[\Gamma]]$. 
Let $I\subset \Lambda$ be the augmentation ideal and set 
$$Q^k:=I^k/I^{k+1}.$$
Let $E$ be an elliptic curve defined over $K$ which has good ordinary reduction at each place lying above $p$. Let ${\rm Sel}_{p^n}(E/K)$ be the usual $p^n$-Selmer group for $E$ over $K$ and set
$$S_p(E/K):=\varprojlim_n {\rm Sel}_{p^n}(E/K).$$
Let ${\rm Tam}(E/K)$ be the product of Tamagawa factors of $E$. For a finite place $v$ of $K$, let $\FF_v$ denote the residue field of $v$. For place $v$ lying above $p$, let $\widetilde E_v$ be the reduction of $E$ modulo $v$. As in \cite[\S 1.1]{BD der}, we make the following assumptions.

\begin{hypothesis}\label{hyp}\
\begin{itemize}
\item[(i)] $p$ does not divide ${\rm Tam}(E/K) \cdot \prod_{v\mid p} \# \widetilde E_v(\FF_v)$. 
\item[(ii)] The image of the Galois representation
$$\rho: G_K \to {\rm Aut}(E[p])\simeq {\rm GL}_2(\FF_p)$$
contains a Cartan subgroup. 
\end{itemize}
\end{hypothesis}

Under Hypothesis \ref{hyp}, Bertolini-Darmon defined a filtration
$$S_p(E/K)=S_p^{(1)}\supset S_p^{(2)}\supset \cdots \supset S_p^{(p)}$$
and derived $p$-adic height pairings
\begin{equation}\label{BD intro}
\langle \cdot,\cdot \rangle_k^{\rm BD}: S_p^{(k)}\times S_p^{(k)} \to Q^k
\end{equation}
for $1\leq k\leq p-1$ (see \cite[\S 2.4]{BD der}). (When $k=1$, this is the usual $p$-adic height pairing.) 

In \cite[\S 11.5]{nekovar}, Nekov\'a\v{r} considered an analogue of derived $p$-adic heights. 
Let
\begin{equation}\label{iw selmer}
C_\infty:=\widetilde \rgamma_{f,{\rm Iw}}(K_\infty/K,T)
\end{equation}
be the Iwasawa-Selmer complex with respect to Greenberg's local conditions (as in \cite[\S 9.6.1]{nekovar}), where $T:=T_p(E)$ denotes the $p$-adic Tate module of $E$. (See Remark \ref{rem iw} below for the definition.) 
The $I$-adic filtration on $C_\infty$ gives a spectral sequence starting with $E_1^{i,j}=H^{i+j}(I^i C_\infty/I^{i+1}C_\infty)$. We call the differential
$$\beta^{(k)}:=d_k^{0,1}: E_k^{0,1}\to E_k^{k,2-k}$$
the $k$-th ``derived Bockstein map". (When $k=1$, this is the usual Bockstein map.) 
There is a natural ``cup product"
$$\cup_k: E_k^{k,2-k}\times E_k^{0,1} \to Q^k.$$
Nekov\'a\v{r}'s derived $p$-adic height pairing 
\begin{equation}\label{nek intro}
\langle \cdot,\cdot \rangle_k^{\rm Nek}: E_k^{0,1}\times E_k^{0,1} \to Q^k
\end{equation}
is defined by 
$$\langle x,y \rangle_k^{\rm Nek}:=\beta^{(k)}(x)\cup_k y.$$

\begin{theorem}\label{main}
Assume Hypothesis \ref{hyp}. Then, for any $1\leq k\leq p-1$, we have a natural identification $E_k^{0,1}=S_p^{(k)}$ and an equality
$$\langle \cdot,\cdot \rangle_k^{\rm BD} = -\langle \cdot,\cdot\rangle_k^{\rm Nek}.$$
\end{theorem}


\begin{remark}
While we need Hypothesis \ref{hyp} to define the Bertolini-Darmon pairing $\langle \cdot,\cdot\rangle_k^{\rm BD}$, Nekov\'a\v{r}'s pairing $\langle \cdot,\cdot\rangle_k^{\rm Nek}$ can be defined without this hypothesis. Moreover, $k\geq 1$ is arbitrary in Nekov\'a\v{r}'s definition. 
\end{remark}

\subsection{Derived $p$-adic heights of Bertolini-Darmon}\label{sec review}

We recall the definition of the (``modulo $p^n$ version" of the) derived $p$-adic height pairing defined by Bertolini-Darmon in \cite[Theorem 2.7]{BD der}. 

Fix a positive integer $n$. Let $L:=K_n$ be the $n$-th layer of $K_\infty/K$. We set
$$G:=\Gal(L/K) \text{ and } \cR:=\ZZ/p^n\ZZ[G]$$
For an $\cR$-module $M$ we set
$$M^\ast:=\Hom_\cR(M,\cR).$$
Let $\cI \subset \cR$ be the augmentation ideal and set
$$\cQ^k:=\cI^k/\cI^{k+1}.$$



\begin{lemma}
Assume Hypothesis \ref{hyp}. Then the restriction map induces an isomorphism
$${\rm Sel}_{p^n}(E/K)\xrightarrow{\sim}{\rm Sel}_{p^n}(E/L)^G.$$
\end{lemma}

\begin{proof}
See \cite[Lemma 1.4]{BD der}. 
\end{proof}

In the following, we assume Hypothesis \ref{hyp} and identify ${\rm Sel}_{p^n}(E/K)$ with ${\rm Sel}_{p^n}(E/L)^G$. For $1\leq k \leq p-1$, we set
\begin{equation}\label{seln}
{\rm Sel}^{(k)}= {\rm Sel}^{(k)}_n:= {\rm Sel}_{p^n}(E/K)\cap \cI^{k-1} {\rm Sel}_{p^n}(E/L).
\end{equation}
We shall define the $k$-th derived $p$-adic height pairing
$$\langle \cdot,\cdot \rangle_{k,n}^{\rm BD}: {\rm Sel}^{(k)}\times {\rm Sel}^{(k)} \to \cQ^k.$$

For a finite set $\Sigma$ of finite places of $K$, the ``$\Sigma$-relaxed" Selmer group is defined by
$${\rm Sel}_{p^n}^\Sigma(E/L):=\{a \in H^1(L,E[p^n]) \mid {\rm Res}_w(a) \in E(L_w)/p^n E(L_w), \ \forall w \mid v, \ v\notin \Sigma\}.$$
Following \cite[Definition 1.5]{BD der}, we call $\Sigma$ an ``admissible set" if the following conditions are satisfied.
\begin{itemize}
\item[(i)] $\Sigma$ does not contain any places lying above $p$ and any bad places. 
\item[(ii)] Each $v\in \Sigma$ splits completely in $L$.
\item[(iii)] For each $v\in \Sigma$, we have $E(K_v)/p^n E(K_v)\simeq (\ZZ/p^n\ZZ)^2$.
\item[(iv)] The restriction map ${\rm Sel}_{p^n}(E/K)\to \bigoplus_{v\in \Sigma}E(K_v)/p^n E(K_v)$ is injective. 
\end{itemize}
By a standard argument using the Chebotarev density theorem, one can show that admissible sets exist. 

We set some notations. For a place $v$ of $K$, we set
$$E(L_v):=\bigoplus_{w\mid v} E(L_w) \text{ and }H^i(L_v,-):=\bigoplus_{w\mid v}H^i(L_w,-).$$

\begin{lemma}\label{global duality}
Let $\Sigma$ be an admissible set. 
\begin{itemize}
\item[(i)] The $\cR$-modules
$$ \bigoplus_{v\in \Sigma}E(L_v)/p^nE(L_v)\text{ and }{\rm Sel}_{p^n}^\Sigma(E/L) $$
are free of rank $2 \cdot \# \Sigma$. 
\item[(ii)] There is a canonical exact sequence
\begin{equation}\label{BD seq}
0\to {\rm Sel}_{p^n}(E/L)\to\bigoplus_{v\in \Sigma}E(L_v)/p^nE(L_v)\xrightarrow{\ell}{\rm Sel}_{p^n}^\Sigma(E/L)^\ast\to {\rm Sel}_{p^n}(E/L)^\ast \to 0.
\end{equation}
\end{itemize}
\end{lemma}
\begin{proof}
(i) is proved in \cite[Lemma 1.7]{BD der} and \cite[Theorem 3.2]{BDMT}. (ii) is proved in \cite[Proposition 1.6]{BD der}, but we give a sketch of the proof for the reader's convenience. By the Poitou-Tate global duality, we have an exact sequence
$$0\to {\rm Sel}_{p^n}(E/L)\to {\rm Sel}_{p^n}^\Sigma(E/L) \to \bigoplus_{v\in \Sigma} H^1(L_v, E)[p^n]\to {\rm Sel}_{p^n}(E/L)^\ast.$$
By taking the dual, we obtain an exact sequence
$${\rm Sel}_{p^n}(E/L)\to \bigoplus_{v\in \Sigma}E(L_v)/p^nE(L_v)\xrightarrow{\ell}{\rm Sel}_{p^n}^\Sigma(E/L)^\ast\to {\rm Sel}_{p^n}(E/L)^\ast \to 0.$$
Here we used the local Tate duality isomorphism
$$ \bigoplus_{v\in \Sigma}H^1(L_v, E)[p^n]\simeq \bigoplus_{v\in \Sigma}(E(L_v)/p^nE(L_v))^\ast.$$
It is sufficient to show that the restriction map ${\rm Sel}_{p^n}(E/L)\to \bigoplus_{v\in \Sigma}E(L_v)/p^nE(L_v)$ is injective. But this follows from the assumption that $\Sigma$ is admissible. 
\end{proof}

\begin{remark}\label{rem adm core}
Set $\fn_\Sigma:=\prod_{v\in \Sigma}v$. If $\Sigma$ is admissible, then we see by Lemma \ref{global duality} that $\fn_\Sigma$ is a core vertex (for the Selmer structure $\cF$ attached to $\sF$ as in Example \ref{ex greenberg}) in the sense of Definition \ref{def core}. Moreover, the core rank is zero. 
\end{remark}

We take the exact sequence (\ref{tate seq}) to be (\ref{BD seq}). Then the construction in \S \ref{sec BD height} yields a pairing
$$\langle \cdot,\cdot\rangle_{k,n}^{\rm BD}: {\rm Sel}^{(k)} \times {\rm Sel}^{(k)} \to \cQ^k.$$
(Note that in the present setting we have 
$$\cS={\rm Sel}_{p^n}(E/L),\ \cT={\rm Sel}_{p^n}(E/L),\  X=\bigoplus_{v\in \Sigma}E(L_v)/p^nE(L_v), \text{ and }Y={\rm Sel}_{p^n}^\Sigma(E/L),$$
and so ${\rm Sel}^{(k)}$ coincides with $\cS_0^{(k)}$ and $\cT_0^{(k)}$ in \S \ref{B notation}.) This pairing is the same as the one constructed in \cite[Theorem 2.7]{BD der}.

Note that the pairing (\ref{BD intro}) is defined by a ``limit" of $\langle \cdot ,\cdot \rangle_{k,n}^{\rm BD}$. See \cite[\S 2.4]{BD der} for details. 

\subsection{Derived $p$-adic heights of Nekov\'a\v{r}}\label{sec nek}

In this subsection, we define the derived $p$-adic height pairing of Nekov\'a\v{r}.

We set some notations. 
Let $S_\infty$ (resp. $S_p$) be the set of infinite (resp. $p$-adic) places of $K$. Let $S_{\rm bad}$ be the set of places of $K$ at which $E$ has bad reduction. Take a finite set $S$ of places of $K$ which contains $S_\infty\cup S_p\cup S_{\rm bad}$. 
%
%
%
%
Take the data $\sF=(\sF_v)_{v\in S_p}$ in (\ref{greenberg local}) to be as in Example \ref{ex greenberg} and let
$$C:=\widetilde \rgamma_\sF(K,{\rm Ind}_{L/K}(T))$$
be the associated Selmer complex (see Definition \ref{def sel}). 

We summarize basic properties of the Selmer complex.

\begin{proposition}\label{prop basic}
Assume Hypothesis \ref{hyp}. Then we have the following. 
\begin{itemize}
\item[(i)] The complex $C$ is quasi-isomorphic to the ``classical Selmer complex" ${\rm SC}_{p}(E_{L/K})$ defined in \cite[Definition 2.3]{BMC}. 
\item[(ii)] The complex $C$ is a perfect complex of $\ZZ_p[G]$-modules and acyclic outside degrees one and two. 
\item[(iii)] We have canonical isomorphisms
$$H^1(C)\simeq S_p(E/L):=\varprojlim_m {\rm Sel}_{p^m}(E/L)$$
and 
$$H^2(C)\simeq {\rm Sel}_{p^\infty}(E/L)^\vee,$$
where ${\rm Sel}_{p^\infty}(E/L):= \varinjlim_m {\rm Sel}_{p^m}(E/L)$ and $(\cdot)^\vee$ denotes the Pontryagin dual. In particular, $H^1(C)$ is $\ZZ_p$-torsion-free.
\item[(iv)] We have the base change property:
$$C_0:= C\lotimes_{\ZZ_p[G]}\ZZ_p \simeq \widetilde \rgamma_\sF(K,T).$$
\item[(v)] The complex $C$ is self-dual, i.e., there is a canonical isomorphim
$$\rhom_{\ZZ_p[G]}(C, \ZZ_p[G])[-3]\simeq C^\iota.$$
Here $C^\iota$ denotes the same complex $C$ on which $G$ acts via the involution $\iota: G\to G; \ \sigma\mapsto \sigma^{-1}$.  
\end{itemize}
\end{proposition}

\begin{proof}
(i) Note that the classical Selmer complex ${\rm SC}_p(E_{L/K})$ is defined by taking the local condition
$$U_v^+(T):= E(L_v)^\wedge[-1]$$
in \cite[\S 6.1]{nekovar}, where $(\cdot)^\wedge$ denotes the $p$-completion. 
Since $\widetilde E_v(\FF_v)[p]=0$ for $v\mid p$, we have
$$\rgamma(L_v,T_v^+)\simeq E(L_v)^{\wedge}[-1]$$
by \cite[Proposition 2.5]{greenberg}. Also, $p\nmid {\rm Tam}(E/K)$ implies
$$\rgamma_{\rm ur}(L_v,T)\simeq E(L_v)^\wedge[-1]$$
for $v \nmid p$. This proves the claim. 

(ii) Note that the hypotheses (${\rm H}_1$)--(${\rm H}_5$) in \cite[Chap. 6]{BMC} are satisfied for $L/K$ by Hypothesis \ref{hyp}. 
The first two claims follow from \cite[Proposition 6.3(i) and (ii)]{BMC} or \cite[Proposition 9.7.7(iii)]{nekovar}. We note that Hypothesis \ref{hyp} also implies that $E(L)[p]=0$ by \cite[Lemma 2.12]{BDMT} so, as in Remark  \ref{HypimplyHyp}, Hypothesis \ref{hyp tech} is also satisfied. In particular, since the $\ZZ_p$-module $H^1(G_{K,S},T)$ is torsion-free by Lemma \ref{tflemma}(iii) and $S_p(E/L)$ is a submodule of $H^1(G_{K,S},T)$, the final claim is also valid.

(iii) This follows from \cite[Proposition 6.3(iv)]{BMC}. 

(iv) See \cite[Proposition 6.3(iii)]{BMC} or \cite[Proposition 9.7.3(i)]{nekovar}.

(v) Since $p\nmid {\rm Tam}(E/K)$, this follows from \cite[Proposition 9.7.3(iv)]{nekovar}. (We identify $T^\ast(1)$ with $T$ by the Weil pairing.)
\end{proof}

\begin{remark}\label{rem iw}
The ``Iwasawa-Selmer complex" in (\ref{iw selmer}) is defined to be 
$$\widetilde \rgamma_{f,{\rm Iw}}(K_\infty/K,T):=\varprojlim_n \widetilde \rgamma_\sF(K,{\rm Ind}_{K_n/K}(T)).$$
This has analogous properties as in Proposition \ref{prop basic}. In particular, it is a perfect complex of $\Lambda$-modules, acyclic outside degrees one and two, and self-dual. 
\end{remark}

Since $C$ is perfect, acyclic outside degrees one and two with $\QQ_p\otimes_{\ZZ_p}H^1(C)\simeq\QQ_p\otimes_{\ZZ_p}H^2(C)$ and $H^1(C)$ is $\ZZ_p$-free by Proposition \ref{prop basic}(ii) and (iii), a standard argument shows that, up to quasi-isomorphism, $C$ is of the form
$$C=\left[ \ZZ_p[G]^d\to \ZZ_p[G]^d \right]$$
for some natural number $d$. We set
$$C/p^n :=C\lotimes_{\ZZ_p} \ZZ/p^n \ZZ.$$
We apply the construction in \S \ref{sec bock pairing} to the exact sequence arising from $C/p^n$:
$$0\to H^1(C/p^n) \to \ZZ/p^n\ZZ[G]^d \to \ZZ/p^n\ZZ[G]^d \to H^2(C/p^n)\to 0.$$
Namely, we take $\cS:=H^1(C/p^n)$, $X:=\ZZ/p^n\ZZ[G]^d$, $Y:=(\ZZ/p^n\ZZ[G]^d)^\ast$, and $\cT:=H^2(C/p^n)^\ast$ in (\ref{tate seq}).  
We define
$$\langle \cdot,\cdot\rangle_{k,n}^{\rm Nek}: H^1(C_0/p^n)^{(k)} \times H^2(C_0/p^n)^{\ast,(k)} \to \cQ^k.$$
to be the Bockstein pairing $\langle \cdot,\cdot \rangle_k^{\rm Boc}$ in \S \ref{sec bock pairing} for the complex $C/p^n$. This is the ``modulo $p^n$ version" of the pairing (\ref{nek intro}).

%
%
 
 \subsection{Proof of Theorem \ref{main}}\label{sec final}

We now prove Theorem \ref{main}. The key is to use Theorems \ref{compari} and \ref{qithm}. 

We set
$$C_{{\rm PT},\Sigma}:=\left[  {\rm Sel}_{p^n}^\Sigma(E/L) \to \bigoplus_{v\in \Sigma} H^1(L_v, E)[p^n] \right].$$
Note that $H^1(C_{{\rm PT},\Sigma})={\rm Sel}_{p^n}(E/L)$ and $H^2(C_{{\rm PT},\Sigma})={\rm Sel}_{p^n}(E/L)^\ast$ by (\ref{BD seq}). 
Note also that, by the proof of Proposition \ref{prop canonical selmer}, $C_{{\rm PT},\Sigma}$ coincides with $C_{{\rm PT},\fn_\Sigma}$ in Proposition \ref{prop core}(ii) for the Selmer structure $\cF$ attached to $\sF$, where $\fn_\Sigma:=\prod_{v\in \Sigma}v$. 
By Remark \ref{rem symm} and Theorem \ref{compari}, the pairing $\langle \cdot, \cdot\rangle_{k,n}^{\rm BD}$ defined in \S \ref{sec review}  coincides with the pairing $\langle \cdot, \cdot \rangle_{k}^{\rm Boc}$ in \S \ref{sec bock pairing} for $C_{{\rm PT},\Sigma}$. 

By Theorem \ref{qithm}, we know that $C_{{\rm PT},\Sigma}$ is quasi-isomorphic to $C/p^n$. (Note that Hypotheses \ref{hyp tech} and \ref{hyp core} are satisfied by Hypothesis \ref{hyp} and Remark \ref{rem adm core}.) 
Since the pairing $\langle \cdot,\cdot\rangle_{k,n}^{\rm Nek}$ is defined to be $\langle \cdot, \cdot \rangle_{k}^{\rm Boc}$ for $C/p^n$, we have
$$\langle \cdot, \cdot\rangle_{k,n}^{\rm BD} = - \langle \cdot, \cdot\rangle_{k,n}^{\rm Nek}.$$
Here the sign is due to the second commutative diagram in Theorem \ref{qithm}. Since $n$ is arbitrary, this proves Theorem \ref{main}. \qed

\appendix

\section{Bockstein maps}\label{sec bock}

In this appendix, we study Bockstein maps in a general algebraic setting. We consider two generalizations of Bockstein maps: ``derived" and ``generalized" Bockstein maps (see Definitions \ref{def der bock} and \ref{def gen bock}). We relate them in Proposition \ref{prop relate}. 

\subsection{Derived and generalized Bockstein maps}
Let $\cR$ be a (possibly non-commutative) ring. By an $\cR$-module we mean a left $\cR$-module. Let $C = (C^i, d^i)$ be a bounded (cochain) complex of $\cR$-modules. 
Fix a two-sided ideal $\cI \subset \cR$. Then we have a decreasing filtration $\{\cI^i C\}_i$ of $C$:
$$C=\cI^0 C \supset \cI C \supset \cI^2 C\supset \cdots.$$
This filtration defines a spectral sequence starting with
$$E_1^{i,j}= H^{i+j}(\cI^iC/\cI^{i+1}C).$$
Explicitly, this is defined as follows. We set
$$Z_k^{i,j}:= \ker(\cI^i C^{i+j} \xrightarrow{d^{i+j}} C^{i+j+1}/\cI^{i+k} C^{i+j+1}),$$
$$B_k^{i,j}:= \cI^i C^{i+j} \cap d^{i+j-1}(\cI^{i-k} C^{i+j-1}),$$
$$E_k^{i,j} := Z_k^{i,j}/(Z_{k-1}^{i+1,j-1} + B_{k-1}^{i,j}).$$
The differential $d_k^{i,j}: E_k^{i,j}\to E_k^{i+k,j+1-k}$ is defined to be the map induced by $d^{i+j}$. One can check that $E_1^{i,j}\simeq H^{i+j}(\cI^iC/\cI^{i+1}C)$ and $E_{k+1}^{i,j}\simeq \ker d_k^{i,j}/\im d_k^{i-k, j+k-1}$. For details, see \cite[\S 2.2]{mccleary} for example. 

{\it In the following, we assume that $C$ is of the form}
\begin{equation}\label{Cform}
C=[ C^1 \xrightarrow{d}C^2].
\end{equation}
In particular, we have
$$H^i(C)=\begin{cases}
\ker d &\text{ if $i=1$,}\\
\coker d &\text{ if $i=2$},\\
0 &\text{ if $i\neq 1,2$}.
\end{cases}$$
Note also that, since $C^{3}=0$, the map $d_k^{k,2-k}$ is zero for any $k\geq 1$ and so $E_k^{k,2-k}$ is a quotient of $E_1^{k,2-k}=H^2(\cI^kC/\cI^{k+1}C)$.

As in \cite{sanoderived}, we give the following definition. 

\begin{definition}\label{def der bock}
We define the {\it $k$-th derived Bockstein map} to be the differential
$$\beta^{(k)}:=d_k^{0,1}: E_k^{0,1}\to E_k^{k,2-k}.$$
\end{definition}

\begin{remark}
When $k=1$, the map $\beta^{(1)}$ coincides with the usual Bockstein map
$$E_1^{0,1}= H^{1}(C/\cI C) \to H^2(\cI C/ \cI^2 C),$$
which is defined to be the connecting homomorphism associated to the short exact sequence
$$0\to \cI C/ \cI^2 C\to C/\cI^2 C \to C/\cI C\to 0.$$
\end{remark}

In \cite{LLSWW}, another generalization of Bockstein maps is considered. In our abstract setting, it is given as follows. 

\begin{definition}\label{def gen bock}
We define the {\it $k$-th generalized Bockstein map} to be the connecting homomorphism
$$\psi^{(k)} := \delta: H^{1}(C/\cI^k C) \to H^2(\cI^k C/\cI^{k+1} C)$$
associated to the short exact sequence
$$0\to \cI^k C/\cI^{k+1} C \to C/\cI^{k+1} C\to C/\cI^k C\to 0.$$
\end{definition}

Clearly, when $k=1$, the map $\psi^{(1)}$ coincides with the Bockstein map. Thus the maps $\beta^{(1)}$ and $\psi^{(1)}$ are the same. In general, the maps $\beta^{(k)}$ and $\psi^{(k)}$ are related as follows. 

\begin{proposition}\label{prop relate}
There is a commutative diagram:
$$\xymatrix{
H^{1}(C/\cI^k C) \ar[r]^{\psi^{(k)}} \ar@{->>}[d]_\pi& H^2(\cI^k C/\cI^{k+1}C)\ar@{->>}[d]^\rho \\
E_k^{0,1} \ar[r]_{\beta^{(k)}}& E_k^{k,2-k},
}$$
where $\pi$ is given in Lemma \ref{fil surj} below and $\rho$ is the natural surjection $E_1^{k,2-k}\twoheadrightarrow E_k^{k,2-k}$. 
\end{proposition}

\begin{lemma}\label{fil surj}
For any $k\geq 1$, there is a natural surjection
$$\pi: H^{1}(C/\cI^k C) \twoheadrightarrow E_k^{0,1}.$$
\end{lemma}

\begin{proof}
Since we assume $C$ is of the form (\ref{Cform}), we have $B_{k-1}^{0,1}=0$. 
Hence we have
\begin{equation}
E_k^{0,1} = Z_k^{0,1}/Z_{k-1}^{1,0}= \{a \in C^{1} \mid da \in \cI^k C^2\}/\{a\in \cI C^{1} \mid da \in \cI^k C^2\} .\label{eexplicit}
\end{equation}
Consider the following commutative diagram with exact rows:
$$
\xymatrix{
0 \ar[r]& \cI C^{1}\ar[r] \ar[d]^d& C^{1} \ar[r]\ar[d]^d& C^{1}/\cI C^{1} \ar[r]\ar[d]^d& 0\\
0 \ar[r]& \cI C^2/\cI^k C^2 \ar[r]& C^2/\cI^k C^2 \ar[r]& C^2/\cI C^2\ar[r] & 0.
}
$$
By the snake lemma, we obtain an exact sequence
$$0\to \{a\in \cI C^{1} \mid da \in \cI^k C^2\}\to \{a\in C^{1} \mid da \in \cI^k C^2\} \to H^{1}(C/\cI C) \xrightarrow{\delta} H^2(\cI C/\cI^k C).$$
Hence we have
\begin{equation}\label{delta surj}
E_k^{0,1}=\ker \delta.
\end{equation}

On the other hand, the short exact sequence
$$0\to \cI C/\cI^k C\to C/\cI^k C\to C/\cI C\to 0$$
induces a long exact sequence
$$\cdots \to H^{1}(C/\cI^k C) \to H^{1}(C/\cI C)\xrightarrow{\delta} H^2(\cI C/\cI^k C)\to \cdots.$$
From this, we obtain a surjection
$$H^{1}(C/\cI^k C)\twoheadrightarrow \ker \delta.$$
Combining this with (\ref{delta surj}), we obtain the desired surjection $H^{1}(C/\cI^k C)\twoheadrightarrow E_k^{0,1}$. 
\end{proof}

\begin{remark}\label{rem special}
The proof of Lemma \ref{fil surj} shows that
$$E_k^{0,1}=\ker \delta= \im (H^{1}(C/\cI^k C) \to H^{1}(C/\cI C)).$$
\end{remark}

We now prove Proposition \ref{prop relate}. 

\begin{proof}[Proof of Proposition \ref{prop relate}]
By the description (\ref{eexplicit}), any element of $E_k^{0,1}$ is represented by an element of $Z_k^{0,1}=\{a \in C^{1} \mid da \in \cI^k C^2\}$. Also, note that $E_k^{k,2-k}$ is a quotient of $\cI^k C^2$. 

Take an element $\overline a \in E_k^{0,1}$ with $a \in Z_k^{0,1}$. By the definition of $\beta^{(k)}$, we have
$$\beta^{(k)}(\overline a) = \overline {da},$$
where $d=d^{1}: Z_k^{0,1}\to \cI^k C^2$ is the differential and $\overline {da}$ is the image of $da \in \cI^k C^2$ in $E_k^{k,2-k}$. Let $a_k \in C^{1}/\cI^k C^{1}$ be the image of $a\in C^{1}$. Note that $a_k \in H^1(C/\cI^k C )$. By the proof of Lemma \ref{fil surj}, we have $\pi( a_k)=\overline a$. By the definition of $\psi^{(k)}$, we have
$$\psi^{(k)}( a_k) = (da)_k,$$
where $(da)_k \in H^2(\cI^kC/ \cI^{k+1}C)$ denotes the image of $da \in \cI^k C^2$. Since we have $\rho((da)_k)= \overline {da}$, we have 
$$\rho \circ \psi^{(k)}( a_k) =\rho((da)_k)=\overline {da}= \beta^{(k)}(\overline a)= \beta^{(k)}\circ \pi ( a_k).$$
Hence we have $\rho \circ \psi^{(k)}=\beta^{(k)} \circ \pi$. This completes the proof. 
\end{proof}

\subsection{Cokernels of Bockstein maps}\label{sec coker}

We study the cokernels of the derived Bockstein map $\beta^{(k)}$ and of the generalized Bockstein map $\psi^{(k)}$. 

The following is proved in \cite[Theorem 2.2.4]{LLSWW}. 

\begin{proposition}\label{coker psi}
For any $k\geq 1$, we have a canonical isomorphism
$$\coker \psi^{(k)}\simeq \cI^k H^2(C)/\cI^{k+1} H^2(C).$$
\end{proposition}

\begin{proof}
By the short exact sequence
$$0\to \cI^k C/\cI^{k+1} C \to C/\cI^{k+1} C\to C/\cI^k C\to 0,$$
we obtain a long exact sequence
$$H^1(C/\cI^k C)\xrightarrow{\psi^{(k)}} H^2(\cI^k C/\cI^{k+1}C)\to H^2(C/\cI^{k+1} C)\to H^2(C/\cI^k C)\to 0.$$
Hence we have
$$\coker \psi^{(k)}\simeq \ker \left(H^2(C/\cI^{k+1} C)\to H^2(C/\cI^k C)\right).$$
Note that for any $i\geq 1$ we have
\begin{eqnarray*}
H^2(C/\cI^i C) &=& \coker \left( C^1/\cI^i C^1\xrightarrow{d^1} C^2/\cI^i C^2\right)\\
&\simeq& \cR/\cI^i \otimes_\cR \coker \left(C^1\xrightarrow{d^1} C^2\right) \\
&\simeq & H^2(C)/\cI^i H^2(C).
\end{eqnarray*}
Therefore, we have
$$\coker \psi^{(k)}\simeq \ker \left(H^2(C)/\cI^{k+1}H^2(C)\to H^2(C)/\cI^k H^2(C)\right) = \cI^k H^2(C)/\cI^{k+1}H^2(C).$$
\end{proof}

It turns out that $\coker \beta^{(k)}$ is the same as $\coker \psi^{(k)}$. 

\begin{proposition}\label{coker beta}
For any $k\geq 1$, we have a canonical isomorphism
$$\coker \beta^{(k)}\simeq \cI^k H^2(C)/\cI^{k+1} H^2(C).$$
\end{proposition}

\begin{proof}
Since $\beta^{(k)}=d_k^{0,1}$ by definition, we have
$$\coker \beta^{(k)} \simeq E_{k+1}^{k,2-k}.$$
By the explicit description of the spectral sequence, we have
\begin{eqnarray*}
E_{k+1}^{k,2-k} &=& Z_{k+1}^{k,2-k}/(Z_k^{k+1, 1-k} + B_k^{k,2-k}) \\
&=& \cI^k C^2/(\cI^{k+1} C^2 + \cI^k C^2\cap d(C^{1})) \\
&=& \coker (\cI^{k+1} C^2 \to \cI^k C^2/ \cI^k C^2 \cap d(C^{1})).
\end{eqnarray*}
Note that
$$\cI^k C^2/ \cI^k C^2 \cap d(C^{1}) \simeq \im (\cI^k C^2 \to C^2/d(C^{1}))= \cI^k H^2(C).$$
Hence we have
$$E_{k+1}^{k,2-k}\simeq \cI^k H^2(C)/\cI^{k+1}H^2(C).$$
\end{proof}

The following is proved in \cite[Proposition 2.10]{sanoderived}. Our approach gives another proof.  

\begin{corollary}\label{cor sano}
Assume that $\cR$ is a discrete valuation ring with maximal ideal $\cI$. Let $\cF:=\cR/\cI$ be the residue field. 
We set
$$\tau_k:= \dim_\cF(E_{k+1}^{k,2-k}) \text{ and }k_0:=\min\{ k\geq 1 \mid \tau_i=\tau_{i+1} \text{ for all $i\geq k$}\}.$$
Then we have
$$\tau_0\geq \tau_1\geq \cdots \geq \tau_{k_0} =\tau_{k_0+1}= \cdots$$
and 
$$H^2(C) \simeq \cR^{ \tau_{k_0}}\times \prod_{i=1}^{k_0} (\cR/\cI^i)^{\tau_{i-1}-\tau_{i}}.$$
\end{corollary}

\begin{proof}
Since there is a natural surjection
$$E_{k}^{k-1,2-(k-1)}\simeq \cI^{k-1} H^2(C)/\cI^{k}H^2(C) \twoheadrightarrow  \cI^k H^2(C)/\cI^{k+1}H^2(C) \simeq E_{k+1}^{k,2-k},$$
we have $\tau_{k-1}\geq \tau_k$. 

By the structure theorem, we can write
$$H^2(C)\simeq \cR^s \times \prod_{i=1}^l (\cR/\cI^i)^{f_i}$$
for some $s\geq 0$, $f_i\geq 0$ and $l\geq 1$. We take $l$ to be minimal. 
By Proposition \ref{coker beta}, we have
$$\tau_k= \dim_\cF( \cI^k H^2(C)/\cI^{k+1}H^2(C)) = s+ f_{k+1} +f_{k+2}+\cdots + f_l. $$
Hence we have
$$\tau_{i-1}-\tau_i = f_i$$
for any $1\leq i\leq l$. We also have $\tau_l= s$. By the minimality of $l$, we have $l=k_0$. This completes the proof. 
\end{proof}

\section{Abstract derived heights}\label{sec abs}

In this appendix, we define two pairings $\langle \cdot,\cdot\rangle_k^{\rm BD}$ and $\langle \cdot,\cdot\rangle_k^{\rm Boc}$ in an abstract setting. The former is a direct generalization of the construction by Bertolini-Darmon in \cite{BD der} (see \S \ref{sec BD height}). The latter is defined by using derived Bockstein maps, based on the idea of Nekov\'a\v{r} in \cite[\S 11.5]{nekovar} (see \S \ref{sec bock pairing}). We prove that these pairings coincide (see Theorem \ref{compari}).

\subsection{Notation}\label{B notation}
We set notations used throughout this appendix. 
Let $p$ be a prime number. Fix a positive integer $n$. Let $G$ be a cyclic group of order $p^n$. We set
$$\cR:=\ZZ/p^n\ZZ[G] \text{ and }\cR_0:=\ZZ/p^n\ZZ.$$
For an $\cR$-module $M$, we set
$$M^\ast:=\Hom_\cR(M,\cR) \text{ and }M_0:=M^G.$$
We sometimes identify $M^\ast$ with $\Hom_{\cR_0}(M,\cR_0)$ via the isomorphism
$$ \Hom_{\cR_0}(M,\cR_0)\xrightarrow{\sim} M^\ast; \ f \mapsto \sum_{\sigma \in G}f(\sigma(-))\sigma^{-1}.$$

Suppose that the following data are given:
\begin{itemize}
\item $\cS,\cT$: finitely generated $\cR$-modules;
\item $X,Y$: free $\cR$-modules of finite rank;
\item an exact sequence
\begin{equation}\label{tate seq}
0\to \cS\to X\xrightarrow{\ell} Y^\ast \to \cT^\ast\to 0.
\end{equation}
\end{itemize}

Let $\cI:=\ker (\cR \twoheadrightarrow \cR_0)$ be the augmentation ideal. 
For $k\geq 1$, we set
$$\cQ^k:=\cI^k/\cI^{k+1}, \ \cS_0^{(k)}:= \cS_0\cap \cI^{k-1}\cS \text{ and }\cT_0^{(k)}:= \cT_0\cap \cI^{k-1} \cT.$$
The following fact is frequently used. 

\begin{lemma}\label{lemQ}
For $1\leq k \leq p-1$, we have
$$\cQ^k\simeq \cR_0.$$
\end{lemma}

\begin{proof}
See \cite[Lemma 2.2]{BD der}. 
\end{proof}

\subsection{The Bertolini-Darmon pairing}\label{sec BD height}

The aim of this subsection is to define a pairing
\begin{equation}\label{BD pairing}
\langle \cdot,\cdot\rangle_k^{\rm BD}:\cS_0^{(k)} \times \cT_0^{(k)} \to \cQ^k
\end{equation}
for $1\leq k \leq p-1$, following the idea of Bertolini-Darmon \cite{BD der}. 

Fix a generator $\gamma \in G$. We define a ``derivative operator" by
$$D^{(k)}:= (-1)^k \sum_{i=0}^{p^n-1} \binom{i}{k}\gamma^{i-k} \in \cR.$$
Note that $N:=D^{(0)}=\sum_{i=0}^{p^n-1}\gamma^i$ is the norm operator. By computation, one checks that
\begin{equation}\label{der relation}
(\gamma-1)D^{(k)} = D^{(k-1)}.
\end{equation}

\begin{lemma}\label{der lemma}
Let $M$ be a free $\cR$-module. Then for any $1\leq k\leq p-1$ we have
$$D^{(k-1)}M = \ker((\gamma-1)^{k}: M\to M)$$
and 
$$\cI^{k} M = \ker (D^{(k-1)}: M\to M).$$
\end{lemma}
\begin{proof}
See \cite[Lemma 2.3 and Corollary 2.5]{BD der}.
\end{proof}

We shall define the pairing (\ref{BD pairing}). Take $s\in \cS_0^{(k)}$ and $t\in \cT_0^{(k)}$. Then there exist $\widetilde s \in \cS$ and $\widetilde t \in \cT$ such that $(\gamma-1)^{k-1}\widetilde s=s$ and $(\gamma-1)^{k-1}\widetilde t = t$. Since we have
$$(\gamma-1)^k \widetilde s = (\gamma-1)s = 0,$$
Lemma \ref{der lemma} implies that there exists $x_s \in X$ such that
$$D^{(k-1)}x_s=\widetilde s \text{ in } X.$$
Here we regard $\cS\subset X$ by the exact sequence (\ref{tate seq}). Similarly, there exists $y_t\in Y$ such that
$$D^{(k-1)}y_t = \widetilde t \text { in }Y. $$
Again, we regard $\cT\subset Y$ by the dual of (\ref{tate seq}). Let $\ell: X\to Y^\ast=\Hom_\cR(Y, \cR)$ be the map in (\ref{tate seq}). Then we have
$$D^{(k-1)}\cdot \ell(x_s)(y_t)=\ell(x_s)(D^{(k-1)}y_t)= \ell(x_s)(\widetilde t)=0.$$
Hence, by Lemma \ref{der lemma}, we have
$$\ell(x_s)(y_t) \in \cI^k.$$
Define the pairing by
$$\langle s,t\rangle_k^{\rm BD}:= \ell(x_s)(y_t) \text{ (mod $\cI^{k+1}$)}.$$
One checks that this is well-defined: it is independent of the choices of $x_s$, $y_t$, and $\gamma$.

\begin{remark}\label{rem symm}
Consider the dual of (\ref{tate seq}):
$$0\to \cT \to Y \xrightarrow{\ell^\ast} X^\ast \to \cS^\ast \to 0.$$
Using this sequence, we can define a pairing
$$(\cdot,\cdot)_k^{\rm BD}: \cT_0^{(k)}\times \cS_0^{(k)} \to \cQ^k$$
exactly in the same way, i.e., 
$$(t,s)_k^{\rm BD}:=\ell^\ast(y_t)(x_s) \text{ (mod $\cI^{k+1}$)}$$
for $s\in \cS_0^{(k)}$ and $t\in \cT_0^{(k)}$. 
Then we have
$$\langle s,t \rangle_k^{\rm BD} = (t,s)_k^{\rm BD}.$$
This follows from the fact that the map $\ell^\ast: Y\to X^\ast$ is given by $\ell^\ast(y)(x)=\ell(x)(y)$ for $x\in X$ and $y\in Y$. 
\end{remark}

\subsection{The Bockstein pairing}\label{sec bock pairing}

In this subsection, we define a pairing
\begin{equation}\label{nek pairing abs}
\langle \cdot,\cdot\rangle_k^{\rm Boc}: \cS_0^{(k)}\times \cT_0^{(k)} \to \cQ^k
\end{equation}
following the idea of Nekov\'a\v{r} \cite[\S 11.5]{nekovar}. 

We consider the complex
$$C:=[X\xrightarrow{\ell} Y^\ast],$$
where $X$ is placed in degree one. By (\ref{tate seq}), we have
$$H^1(C)=\cS \text{ and }H^2(C)=\cT^\ast.$$
As in Appendix \ref{sec bock}, we consider the spectral sequence $E_k^{i,j}$ associated to $C$ and the filtration $\{\cI^i C\}_i$. 
Let
$$\beta^{(k)}:=d_k^{0,1}: E_k^{0,1} \to E_k^{k,2-k}$$
be the $k$-th derived Bockstein map in Definition \ref{def der bock}.

\begin{lemma}\label{lem spec}\
\begin{itemize}
\item[(i)] There is a canonical isomorphism $E_k^{0,1}\simeq \cS_0^{(k)}$.
\item[(ii)] There is a canonical map $\iota_k: E_k^{k,2-k}\to \Hom_{\cR_0}(\cT_0^{(k)}, \cQ^k)$. 
\end{itemize}
\end{lemma}

\begin{proof}
Let $\gamma \in G$ be a generator. Since $X$ and $Y$ are free $\cR$-modules, the norm operator $N:=\sum_{i=0}^{p^n-1}\gamma^i$ induces an isomorphism
$$N: C/\cI C=[X/\cI X\to Y^\ast/\cI Y^\ast] \xrightarrow{\sim} [X_0 \to (Y^\ast)_0].$$
In particular, it induces an isomorphism
$$N: H^1(C/\cI C) \xrightarrow{\sim} \cS_0.$$
Let $\pi: H^1(C/\cI^k C)\to H^1(C/\cI C)$ be the natural map. 
By Remark \ref{rem special}, we have
$$E_k^{0,1}= \im \pi.$$
To prove (i), it is sufficient to show that
$$N (\im \pi) = \cS_0\cap \cI^{k-1} \cS. $$
We first show that the left hand side is contained in the right hand side. Take an element $a \in \im \pi$. Then there exists $\widetilde a \in H^1(C/\cI^k C)$ such that $\pi(\widetilde a)=a$. Regard $\widetilde a \in X/\cI^k X$ and take its lift $x_a \in X$. Note that the norm map $N: X\to X$ factors through $X/\cI X$. By the relation $(\gamma-1)^{k-1} D^{(k-1)} =N$ (see (\ref{der relation})), we have
$$Na = (\gamma-1)^{k-1}D^{(k-1)} x_a.$$
By Lemma \ref{der lemma}, note that the map $D^{(k-1)}: X\to X$ factors through $X/\cI^k X$. 
Since $\widetilde a$ lies in $H^1(C/\cI^k C)$, one sees that $D^{(k-1)}x_a \in H^1(C)=\cS$. Therefore, we have
$$Na \in (\gamma-1)^{k-1} \cS.$$
This shows that $N (\im \pi) \subset \cS_0\cap \cI^{k-1} \cS$. Conversely, take an element $s \in \cS_0\cap \cI^{k-1}\cS$. Then there exists $\widetilde s \in \cS$ such that $(\gamma-1)^{k-1}\widetilde s =s$. Since we have
$$(\gamma-1)^k \widetilde s = (\gamma-1)s=0,$$
Lemma \ref{der lemma} implies that there exists $x_s \in X$ such that $D^{(k-1)}x_s =\widetilde s$. Let $\overline x_s \in X/\cI^k X$ be the image of $x_s$. Since $D^{(k-1)}\ell(\overline x_s)=\ell(D^{(k-1)}x_s)=\ell(\widetilde s)=0$ and $D^{(k-1)}$ is injective on $Y^\ast/\cI^k Y^\ast$ by Lemma \ref{der lemma}, we have $\overline x_s \in H^1(C/\cI^k C)$. Also, we have $N(\pi(\overline x_s))= (\gamma-1)^{k-1}D^{(k-1)}x_s=(\gamma-1)^{k-1}\widetilde s=s$. This shows that $N (\im \pi) \supset \cS_0\cap \cI^{k-1} \cS$. Thus we have proved the equality $N (\im \pi)= \cS_0\cap \cI^{k-1} \cS$. This completes the proof of (i).

To prove (ii), consider the dual complex
$$D:= [Y\xrightarrow{\ell^\ast} X^\ast].$$
(We identify $Y^{\ast \ast}$ with $Y$.) Then we have $H^1(D)=\cT$. Let $F_k^{i,j}$ be the spectral sequence associated to $D$ and the filtration $\{\cI^i D\}_i$. Then (i) implies that
$$F_k^{0,1}\simeq \cT_0^{(k)}.$$
On the other hand, as in \cite[\S 11.5.3]{nekovar}, we have a canonical cup product 
$$\cup_k:E_k^{k,2-k} \times F_k^{0,1}\to \cQ^k.$$
This induces a map
$$E_k^{k,2-k}\to \Hom_{\cR_0}(F_k^{0,1}, \cQ^k); \ a\mapsto (b\mapsto a\cup_k b).$$
Combining this with $F_k^{0,1}\simeq \cT_0^{(k)}$, we obtain the map in (ii). 
\end{proof}

By Lemma \ref{lem spec}, we have an identification $E_k^{0,1}= \cS_0^{(k)}$ and a map $\iota_k: E_k^{k,2-k}\to \Hom_{\cR_0}(\cT_0^{(k)}, \cQ^k)$. We define the pairing (\ref{nek pairing abs}) by
$$\langle s,t\rangle_k^{\rm Boc}:= \iota_k(\beta^{(k)}(s) )(t) \in \cQ^k.$$


\subsection{Comparison}

We now show that two pairings defined above are the same.

\begin{theorem}\label{compari}
The pairings $\langle \cdot,\cdot\rangle_k^{\rm BD}$ and $\langle \cdot,\cdot \rangle_k^{\rm Boc}$ coincide. 
\end{theorem}

\begin{proof}
Take $s \in \cS_0^{(k)}$ and $t\in \cT_0^{(k)}$. As in the definition of $\langle \cdot,\cdot\rangle_k^{\rm BD}$, take $\widetilde s \in \cS$, $x_s \in X$, $\widetilde t \in \cT$, and $y_t\in Y$. Let $\overline x_s \in X/\cI^k X$ be the image of $x_s$. Then the proof of Lemma \ref{lem spec}(i) shows that $\overline x_s \in H^1(C/\cI^k C)$ and $\pi(\overline x_s) = s$, where $\pi: H^1(C/\cI^k C)\twoheadrightarrow E_k^{0,1}=\cS_0^{(k)}$ is the map in Lemma \ref{fil surj}. (Note that, to identify $E_k^{0,1}$ with $\cS_0^{(k)}$, we implicitly use the norm map.) Let 
$$\psi^{(k)}: H^1(C/\cI^k C)\to H^2(\cI^k C/\cI^{k+1}C)$$
be the generalized Bockstein map in Definition \ref{def gen bock}. By Lemma \ref{lemQ}, we have identifications
$$H^2(\cI^k C/ \cI^{k+1} C)= H^2(C/\cI C)\otimes_{\cR_0} \cQ^k = \Hom_{\cR_0}(\cT_0, \cQ^k).$$
By Proposition \ref{prop relate}, we have a commutative diagram
$$\xymatrix{
H^1(C/\cI^k C) \ar[r]^{\psi^{(k)}} \ar@{->>}[d]_\pi & \Hom_{\cR_0}(\cT_0,\cQ^k) \ar@{->>}[d]\\
\cS_0^{(k)} \ar[r]_-{\iota_k\circ \beta^{(k)}} & \Hom_{\cR_0}(\cT_0^{(k)},\cQ^k),
}$$
where the right vertical arrow is the restriction map. From this, we have
\begin{equation}\label{nek compute}
\langle s,t\rangle_k^{\rm Boc} := \iota_k(\beta^{(k)}(s))(t)= \psi^{(k)}(\overline x_s)(t).
\end{equation}

On the other hand, the map $\psi^{(k)}$ is by definition the snake map associated to the diagram
$$
\xymatrix{
0\ar[r] & \cI^k X/\cI^{k+1} X \ar[r] \ar[d] & X/\cI^{k+1} X \ar[r]\ar[d] & X/\cI^k X\ar[r] \ar[d] & 0\\
0\ar[r] &\cI^k Y^\ast/\cI^{k+1}Y^\ast \ar[r] & Y^\ast/\cI^{k+1}Y^\ast \ar[r]& Y^\ast/\cI^k Y^\ast \ar[r] &0,
}
$$
where the vertical maps are induced by $\ell : X\to Y^\ast$. Hence the element $\psi^{(k)}(\overline x_s) $ is represented by
$$\overline {\ell(x_s)} \in Y^\ast/\cI^{k+1}Y^\ast=\Hom_\cR(Y, \cR/\cI^{k+1}).$$
The image of the map $\overline {\ell(x_s)}: Y \to \cR/\cI^{k+1}$ actually lies in $\cQ^k$, so it factors through $Y/\cI Y \simeq Y_0$. Since $y_t \in Y$ is a lift of $t \in \cT_0^{(k)}\subset Y_0$, we have
$$\psi^{(k)}(\overline x_s)(t)= \overline {\ell(x_s)}(y_t) =:\langle s,t \rangle_k^{\rm BD}.$$
By (\ref{nek compute}), we have
$$\langle s,t\rangle_k^{\rm Boc}=\langle s,t\rangle_k^{\rm BD}.$$
\end{proof}

\end{document}